\documentclass[12pt]{amsart}

\usepackage{amsmath,amssymb,eucal,color,tikz,graphicx,mathrsfs,subcaption,enumerate,hyperref,mathscinet}

\usetikzlibrary{decorations.text}
\usetikzlibrary{calc}
\usetikzlibrary{through}
\usepackage{tkz-euclide}
\usetkzobj{all}

\newdimen\XCoord
\newdimen\TCoord
\definecolor{olive}{rgb}{0.5,0.5,0}

\newcommand{\rev}[1]{{{#1}}}
\newcommand{\union}{\mathop{\bigcup}}
\newcommand{\vL}{\varLambda}
\newcommand{\zout}{z^K_{\mathrm{o}}}
\newcommand{\zin}{z^K_{\mathrm{i}}}

\newcommand{\dist}{\mathop{\mathrm{dist}}}
\newcommand{\diag}{\mathop{\mathrm{diag}}}
\newcommand{\tin}{\tau_{\mathrm{i}}}
\newcommand{\tout}{\tau_{\mathrm{o}}}
\newcommand{\htin}{\hat{\tau}_{\mathrm{i}}}
\newcommand{\htout}{\hat{\tau}_{\mathrm{o}}}
\newcommand{\gi}{g_{\mathrm{i}}}
\newcommand{\go}{g_{\mathrm{o}}}
\newcommand{\hgi}{\hat g_{\mathrm{i}}}
\newcommand{\hgo}{\hat g_{\mathrm{o}}}
\newcommand{\thetao}{\theta_{\!\mathrm{o}}}
\newcommand{\ii}{\hat{\imath}}

\newcommand{\hn}{\hat n}
\newcommand{\hw}{\hat w}

\newcommand{\hx}{{\hat x}}
\newcommand{\htt}{{\hat t}}
\newcommand{\hK}{\hat K}

\newcommand{\hz}{\hat z}

\newcommand{\bW}{\breve W}
\newcommand{\bV}{\breve V}
\newcommand{\bA}{\breve A}
\newcommand{\bD}{\breve D}
\newcommand{\bw}{\breve w}
\newcommand{\bv}{\breve v}
\newcommand{\bu}{\breve u}
\newcommand{\bq}{\breve q}
\newcommand{\bz}{\breve z}
\newcommand{\btin}{\breve{\tau}_{\mathrm{i}}}
\newcommand{\btout}{\breve{\tau}_{\mathrm{o}}}

\newcommand{\rprp}[1]{{#1}^\perp}
\newcommand{\lprp}[1]{\sideset{^\perp}{}{\mathop{#1}}}

\newcommand{\Babuska}{Babu{\v{s}}ka}

\newcommand{\ip}[1]{\langle {#1} \rangle}

\newcommand{\Dd}{\mathscr{D}}

\newcommand{\om}{\varOmega}

\newcommand{\oh}{\varOmega_h}

\def\d{\partial}
\newcommand{\RRR}{\mathbb{R}}
\newcommand{\ZZZ}{\mathbb{Z}}
\newcommand{\RRRmm}{\mathbb{R}^{m\times m}}

\newcommand{\dout}{\d_{\mathrm{o}}}
\newcommand{\din}{\d_{\mathrm{i}}}

\newcommand{\Gout}{\varGamma_{\mathrm{o}}}
\newcommand{\Gin}{\varGamma_{\mathrm{i}}}
\newcommand{\Gio}{\varGamma_{\mathrm{io}}}
\newcommand{\hGio}{\hat\varGamma_{\mathrm{io}}}
\newcommand{\db}{\d_{\mathrm{b}}}
\newcommand{\dc}{\d_{\mathrm{c}}}

\newcommand{\ran}{\mathop{\mathrm{ran}}}

\newcommand{\D}{\mathcal{D}}
\renewcommand{\L}{\mathcal{L}}

\newcommand{\At}{\tilde{A}}
\newcommand{\Wt}{\tilde{W}}
\newcommand{\Vsp}{\rprp{(V^*)}}
\newcommand{\Vp}{\rprp{V}}

\newcommand{\dom}{\mathrm{dom}}

\newtheorem{theorem}{Theorem}[section]
\newtheorem{lemma}[theorem]{Lemma}

\newtheorem{proposition}[theorem]{Proposition}

\theoremstyle{remark}
\newtheorem{remark}[theorem]{Remark}
\newtheorem{definition}[theorem]{Definition}

\numberwithin{equation}{section}

\title[Space-time tent pitching and traces]{
A tent pitching scheme
\\
motivated by Friedrichs theory
}

\author[Gopalakrishnan]{Jay Gopalakrishnan}
\address{PO Box 751, Portland State University, Portland, OR 97207-0751.}

\author[Monk]{Peter Monk}
\address{Department of Mathematics, University of Delaware, Newark, DE.}

\author[Sepulveda]{Paulina Sep{\'u}lveda}
\address{PO Box 751, Portland State University, Portland, OR 97207-0751.}

\thanks{This work was partially supported by the NSF grants 
DMS-1318916 and DMS-1216620 and the AFOSR grant FA9550-12-1-0484.}


\begin{document}

\begin{abstract}
  Certain Friedrichs systems can be posed on Hilbert spaces normed
  with a graph norm.  Functions in such spaces arising from advective
  problems are found to have traces with a weak continuity property at
  points where the inflow and outflow boundaries meet. Motivated by
  this continuity property, an explicit space-time finite element
  scheme of the tent pitching type, with spaces that conform to the
  continuity property, is designed. Numerical results for a model
  one-dimensional wave propagation problem are presented.


\end{abstract}

\maketitle

\graphicspath{{./figs/}}

\section{Introduction}

A commonly used approach for constructing numerical methods to solve
time-dependent problems is {based on the method of lines, where a
  discretization of all space derivatives is followed by a
  discretization of time derivatives.  The resulting methods are
  called implicit or explicit depending on whether one can advance in
  time with or without solving a spatially global problem.  The
  study in this paper targets a different class of methods referred to
  as {\em locally implicit} space-time finite element methods, which
  advance in time using calculations that are local within space-time
  regions of simulation. Examples of such methods are provided by
  ``tent pitching'' schemes, which mesh the space-time region using
  tent-shaped subdomains and advance in time by varying amounts at
  different points in space.}

Ideas to advance a numerical solution in time by local operations
  in space time regions were explored even as early
  as~\cite{Oden69}. Recurrence relations on
  multiple slabs of rectangular space-time elements were considered
  in~\cite{Kaczk75}, whose ideas were generalized to non-rectangular
  space-time elements for beams \rev{ and plates} in~\cite{Bajer86}. These works are not
  so related to the current work as some of the more modern
  references.  Closest in ancestry to the method we shall consider is
  found in~\cite{Richt94a} where it was called explicit space-time
  elements. The space-time discontinuous Galerkin (SDG) method was
  announced almost at the same time in~\cite{LowriRoeLeer95} and
  continues to see active
  development~\cite{MilleHaber08,PalanHaberJerra04,YinAcharSobh00}.
  Against this backdrop, we highlight two papers that brought tent
  pitching ideas into the numerical analysis
  community~\cite{FalkRicht99,MonkRicht05}. The questions we choose to
  ask in this work have been heavily influenced by these two works. We
  should note that the name ``tent pitching'' has been traditionally
  used for meshing schemes that advance a space-time
  front~\cite{ErickGuoySulli05,UngorSheff02}, but in this paper tent
  pitching refers to the discretization scheme together with all the
  required meshing.

\begin{figure}     
  \begin{tabular}{c|c|c}
    \includegraphics[trim=0in 0in 0in 0.7in,clip,width=0.3\textwidth]{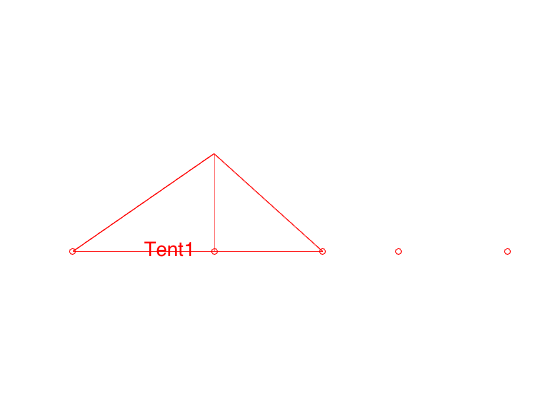}
    & 
    \includegraphics[trim=0in 0in 0in 0.7in,clip,width=0.3\textwidth]{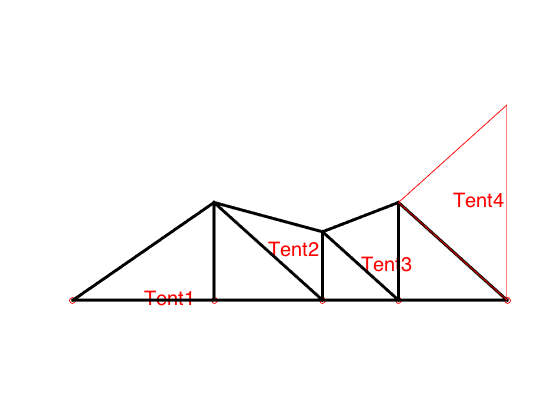}    
    & 
    \includegraphics[trim=0in 0in 0in 0.7in,clip,width=0.3\textwidth]{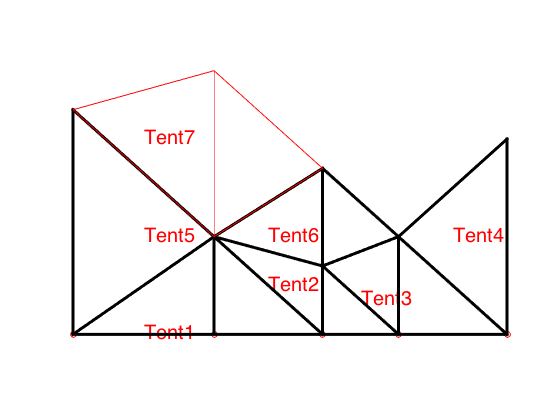}    
    \\
    \includegraphics[trim=0in 0in 0in 0.7in,clip,width=0.3\textwidth]{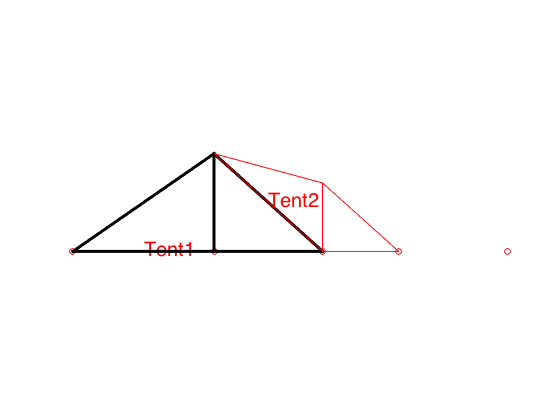}
    & 
    \includegraphics[trim=0in 0in 0in 0.7in,clip,width=0.3\textwidth]{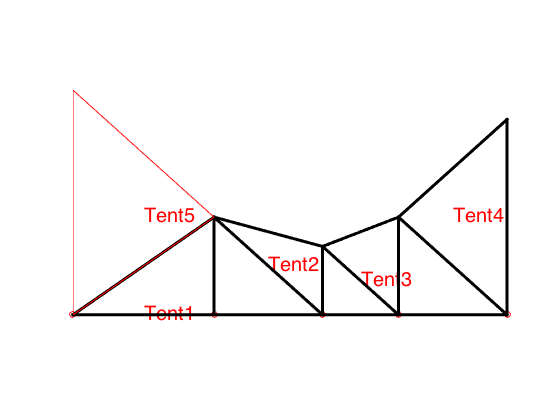}    
    & 
    \includegraphics[trim=0in 0in 0in 0.7in,clip,width=0.3\textwidth]{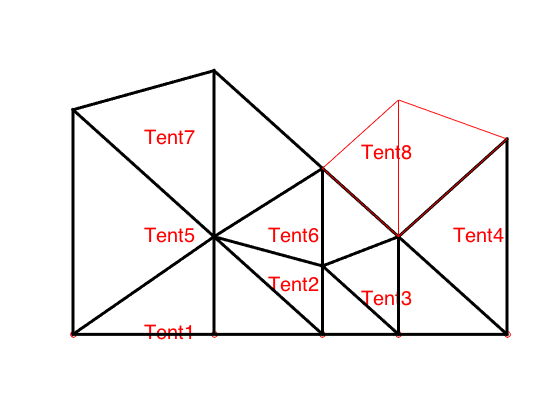}    
    \\
    \includegraphics[trim=0in 0in 0in 0.7in,clip,width=0.3\textwidth]{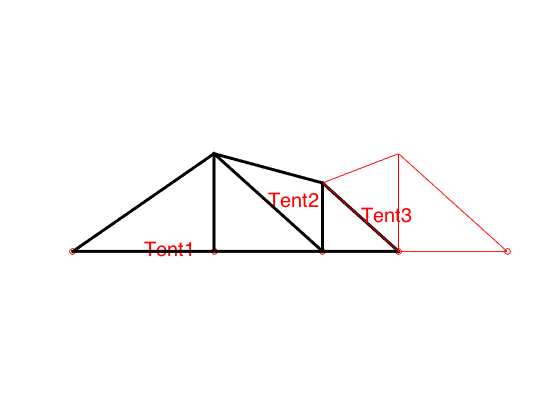}    
    & 
    \includegraphics[trim=0in 0in 0in 0.7in,clip,width=0.3\textwidth]{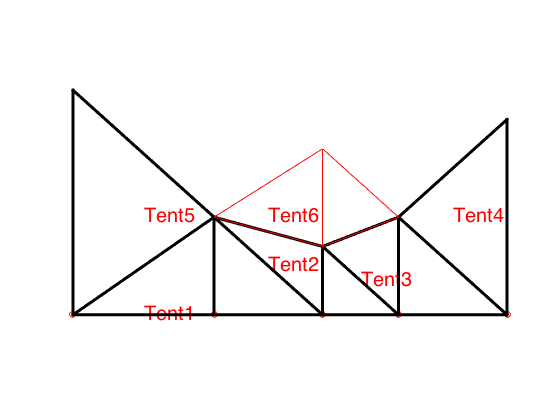}    
    & 
    \includegraphics[trim=0in 0in 0in 0.7in,clip,width=0.3\textwidth]{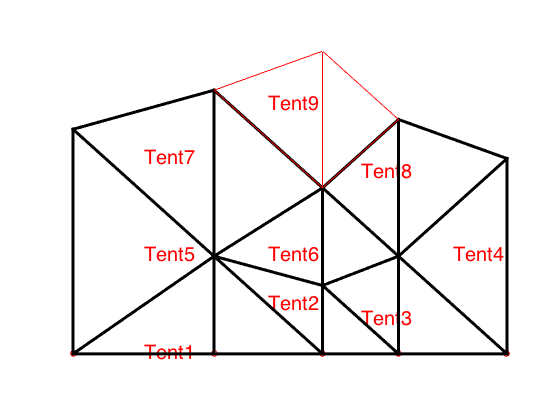}    
  \end{tabular}
  \caption{Tent pitching (read column by column)}
  \label{fig:tents}
\end{figure}

{To give an overview of what is involved in a tent pitching
  scheme}, consider the case of a hyperbolic problem posed in one
space dimension with time as the second dimension. Given a spatial
mesh, we pitch a tent by erecting a tent pole (vertically in time) at
a vertex, as in Figure~\ref{fig:tents}. (Precise definitions of
``tents'' etc.\ appear later -- see Definition~\ref{defn:tent}.)  In
the plots of Figure~\ref{fig:tents}, the horizontal and vertical
dimensions are space and time, respectively.  The height of the tent
pole must be chosen small enough in relation to the hyperbolic
propagation speed, so that the domain of dependence of all points in
the tent remains within the tent's footprint. We then use the given
initial data to solve, by some numerical scheme, the hyperbolic
problem restricted to the tent. Proceeding to the next vertex where
the second tent is pitched in Figure~\ref{fig:tents}, we find that the
initial data combined with the solution in the previous tent, provides
inflow data to solve the hyperbolic problem there. Solution on the
newer tents proceeds similarly. This shows the sense in which tent
pitching schemes are locally implicit: they only involve solving local
problems tent by tent.

{Having explained tent pitching schemes in general, we should now
  emphasize that the main result of this paper is not a new tent
  pitching scheme (although one is included to show relevance).
  Rather, this paper is mainly concerned with answering a few
  theoretical questions {\em motivated} by tent pitching
  schemes. Indeed, our main result is a characterization of traces of
  a Friedrichs space on a tent-shaped domain and builds on the recent
  advances in Friedrichs
  theory~\cite{AntonBuraz10,ErnGuerm06,ErnGuermCapla07,Fried58}.  To
  explain the Friedrichs connection, we should first note that all the
  previous tent pitching schemes} use non-conforming space-time
discontinuous Galerkin discretizations.  Design of tent pitching
methods within a conforming setting, while holding the promise of
locally adaptive time marching with fewer unknowns, pose interesting
questions: What is the weak formulation that the tent pitching scheme
should conform to?  What are the spaces?  What are the finite element
subspaces one should use?  {These questions form the motivation
  for this study and while attempting to answer them, Friedrichs
  spaces and their traces appear naturally, as we shall see.}  While
we are far from answering the above questions for a general Friedrichs
system, our modest aim in this paper is to provide some answers for a
few simple problems in one space dimension.

Accordingly, there are two parts to this paper. {The first and the
  main part of the paper consists of
Sections~\ref{sec:friedrichs}--\ref{sec:examples}. While results of
Sections~\ref{sec:friedrichs} and~\ref{sec:weak-form} are applicable to
any abstract Friedrichs system, Section \ref{sec:examples} focusses
mainly on an advection example and its implications for hyperbolic
systems.  This leads to observations on the traces of certain
Friedrichs spaces. The theory} clarifies a weak continuity property of
the traces at the points where inflow and outflow part of boundaries
(defined precisely later) meet. It is relevant in the tent pitching
context because in the tent-shaped domains used in tent pitching
schemes, inflow and outflow boundaries always meet. The second part of
the paper, consisting of Sections~\ref{sec:wave} and~\ref{sec:scheme},
designs an explicit space-time finite element scheme of the tent
pitching type using the spaces and weak formulations motivated by the
first part. The method we construct is a low order method that works
on unstructured grids.  On uniform grids, comparison with a standard
low order finite difference method does not reveal any striking
advantages for the new method, as we will see in
Section~\ref{sec:numerical}. Yet we hope that this study will pave the
way to a better understanding of conforming tent pitching
discretizations, the spaces involved, and eventually lead to high
order methods on unstructured grids for multidimensional problems. We
begin with some preliminaries on Friedrichs systems in the next
section.

\section{Friedrichs systems} \label{sec:friedrichs}

Our approach is influenced by the modern take on the classical
work of Friedrichs~\cite{Fried58}, as presented
in~\cite{ErnGuerm06,ErnGuerm06a,ErnGuermCapla07}. Let $L$ be a Hilbert
space over $\RRR$ with inner product $(\cdot,\cdot)_L$ and norm $\|
\cdot\|_L$, and let $\D$ be a dense subspace of $L$.  Suppose $A$ and
$\At$ are linear maps from $\D$ into $L$ satisfying
\begin{subequations}
  \label{eq:T}
\begin{gather}
  \label{eq:T1}
  (A\phi,\psi)_L  = (\phi,\At\psi)_L,\qquad  \forall \phi,\psi \in \D,
  \\   \label{eq:T2}
  \exists \; c>0: \quad \| (A + \At) \phi \|_L \le c \| \phi\|_L, \qquad
  \forall \phi \in \D.
\end{gather}
\end{subequations}
Let $W_0$ be the \rev{completion} of $\D$ in the norm $\| \phi \|_W = \left( \|
  \phi\|_L^2 + \| A \phi\|_L^2\right)^{1/2}$. Then, with $L$ as a
pivot Hilbert space, identified to be the same as its dual $L'$, we
have $\D \subseteq W_0 \subseteq L \equiv L' \subseteq W_0'$. It is
now standard to extend $A$ and $\At$ as bounded linear operators from
$W_0$ into $L$, i.e., $A,\At \in \L( W_0,L)$. Moreover, it is also
well-known that Assumption~\eqref{eq:T} implies that they can be further
extended to $A,\At \in \L( L,W_0')$ via
\begin{equation}
  \label{eq:11}
\ip{A \ell, w_0}_{W_0} = (\ell, \At w_0)_L,\qquad
\ip{ \At \ell, w_0}_{W_0} = (\ell, A w_0)_L,
\qquad\forall \ell\in L, \;w_0\in W_0.
\end{equation}
Here and throughout, we use $\ip{\cdot,\cdot}_X$ to denote the duality
pairing in $X$. Next, defining $W = \{ v \in L: A v \in L\}$, we
observe that $W_0\subseteq W$ and that $W$ normed with $\| \cdot\|_W$
defined above is a Hilbert space.  Hereon, the norm on any normed
linear space $X$ will be denoted by $\|\cdot\|_X$.

An important ingredient in Friedrichs theory is the ``boundary''
operator $D$ in $\L(W,W')$ defined by
\begin{equation}
  \label{eq:D}
  \ip{D u,v}_W = (A u,v)_L - (u,\At v)_L  \qquad \forall u,v\in W.
\end{equation}
This is an abstraction of an integration by parts identity. For any
operator $B \in \L(W,W')$, we define $B^* \in \L(W,W')$ by $\ip{B^*
  u,v}_W = \ip{ Bv,u}_W$ for all $u,v\in W$. For subspaces $S \subseteq W$ and $R \subseteq W'$,
define the right and left annihilators by
\begin{align*}
  \rprp S & = \{ w'\in W': \ip{ w',s}_W=0\text{ for all } s\in S\},
  \\
  \lprp R & = \{ w\in W: \ip{ s',w}_W=0\text{ for all }s'\in R\}.
\end{align*}
These results are well known~\cite{ErnGuermCapla07}:

\bigskip

\begin{proposition}
  \label{prop:collect}
  The following are consequences of Assumption~\eqref{eq:T}:
  \begin{enumerate}[(a)]
  \item \label{item:Dsym} $D^* = D$.
  \item \label{item:equiv} The norm $\| w \|_{\Wt} = (\| w\|_L^2 + \|
    \At w \|_L^2)^{1/2}$ {is equivalent to $\| w \|_W$ for all $w \in W$.}
  \item\label{item:ker} $\ker D = W_0$.
  \item\label{item:ran} $\ran D = W_0^\perp$.
  \end{enumerate}
\end{proposition}

We will henceforth tacitly assume~\eqref{eq:T} throughout this
section.  In the traditional Friedrichs theory, another ``boundary
operator'' $M$, also in $\L(W,W')$, plays a leading role. This is a
generalization of certain matrices used by Friedrichs~\cite{Fried58}
to impose boundary conditions.  In the generalization of Friedrichs
theory to the Hilbert space setting, as described
in~\cite{ErnGuermCapla07}, the operator $M$ is assumed to satisfy
\begin{subequations}
  \label{eq:M}
  \begin{gather}
    \ip{ M w, w}_W \ge 0, \quad \forall w\in W,
    \\
    W  = \ker(D-M) + \ker (D+M).
  \end{gather}
\end{subequations}

The theory in~\cite{ErnGuermCapla07} addresses the unique
solvability of two problems: The first is to find a $u\in W,$ given
any $f\in L$, satisfying $ A u = f$ (typically a partial differential
equation), and $ (D-M)u = 0$ (typically a boundary condition).
The second problem is the ``dual'' problem of solving $\At u = f$ satisfying
$(D+M^*)u=0$.  These two problems are uniquely solvable if and only if
the following two conditions hold, respectively:
\begin{subequations}
  \label{eq:A}
  \begin{align}
    \label{eq:A1}
    A &: \ker(D-M) \to L\;  \text{ is a bijection, and }
    \\
    \label{eq:A2}
    \At & : \ker(D+M^*) \to L\; \text{ is a bijection}.
  \end{align}
\end{subequations}
Some sufficient conditions for~\eqref{eq:A} to hold can be found
in~\cite{ErnGuerm06,ErnGuermCapla07}.

In~\cite{ErnGuermCapla07}, an intrinsic approach without the operator
$M$ was discovered. It uses the double cones
\begin{align*}
  C^+ & = \{ w \in W: \ip{ D w,w}_W \ge 0 \},
  \\
  C^- & = \{ w \in W: \ip{ D w,w}_W \le 0 \}.
\end{align*}
The intrinsic approach replaces~\eqref{eq:M} by the following
assumption on two subspaces of $W$ denoted by $V$ and $V^*$:
\begin{subequations}
  \label{eq:V}
  \begin{align}
    \label{eq:V1}
    V & \subseteq C^+, && V^* \subseteq C^-,
    \\
    \label{eq:V2}
    V & = \lprp{D(V^*)}, && V^* = \lprp{D(V)}.
  \end{align}
\end{subequations}
Clearly, \eqref{eq:V2} implies that both $V$ and $V^*$ are closed, and moreover,
\begin{equation}
  \label{eq:3}
  \ker D = W_0 \subseteq V \cap V^*.
\end{equation}
Note that the reflexivity of Hilbert spaces
and~\eqref{eq:V2} imply that
\begin{gather}
\label{eq:10}
\rprp{(V^*)}   = \rprp{ (\lprp{D(V)})} = D(V), \\
\nonumber
\rprp V       = \rprp{ (\lprp{D(V^*)} )} = D(V^*).
\end{gather}

The theory in~\cite{ErnGuermCapla07} provides sufficient conditions
for unique solvability of two problems: The first is to find a $u\in
W,$ given any $f\in L$, satisfying
\begin{subequations}
  \label{eq:Au=f}
\begin{align}
  \label{eq:Au=f:pde}
   A u & = f && \text{(typically a partial differential equation),} \\
   \label{eq:Au=f:bc}
   u & \in V && \text{(typically a boundary condition)}.
\end{align}
\end{subequations}
The second is the ``dual'' problem of solving for a $u \in V^*$
satisfying $\At u = f$.  These two problems are uniquely solvable if
and only if
\begin{subequations}
  \label{eq:A-V}
  \begin{align}
    \label{eq:A1-V}
    A &: V \to L\;  \text{ is a bijection, and }
    \\
    \label{eq:A2-V}
    \At & : V^*\to L\; \text{ is a bijection}.
  \end{align}
\end{subequations}
\rev{A coercivity condition on $A + \At$ is sufficient
  for~\eqref{eq:A-V} to hold, as proved
  in~\cite{ErnGuermCapla07}. However, for operators like the transient
  wave operator considered later, $A + \At$ is zero and cannot be
  coercive.}

\rev{Hence, our first point of departure from~\cite{ErnGuermCapla07}
  is the introduction of another simple sufficient condition for
  unique solvability.  It requires that the operator $A$ be bounded
  from below, a condition which is often easy to verify for
  time-dependent problems (see e.g.~\cite{WieneWohlm13}). Although the
  next theorem requires both $A$ and $\At$ to be bounded below, one of
  these conditions can be easily removed in most applications, as
  detailed in Remark~\ref{rem:closed}.}

\begin{theorem}
  \label{thm:bddbelow}
  Suppose~\eqref{eq:T} and~\eqref{eq:V} hold.  If there is a constant  $c>0$
  such that $A : V \to L$ 
  satisfies
  \begin{subequations}
    \label{eq:AtAbelow}
    \begin{align}
      \label{eq:Abdbelow}
      \| A u \|_L & \ge c \| u \|_W  && \forall u \in V,
      \text{ and}
      \\
      \| \At u \|_L & \ge c \| u \|_W  && \forall u \in V^*,
      \label{eq:Atbdbelow}
    \end{align}
  \end{subequations}
  then~\eqref{eq:A-V} holds.
\end{theorem}
\begin{proof}
Inequality~\eqref{eq:Abdbelow} implies that $A: V \to L $ is
  injective and has closed range. Hence $A$ is a bijection if its
  adjoint is injective, i.e., if
  \begin{equation}
    \label{eq:adjinj}
    \{ \ell \in L: \;
    (Av,\ell)_L =0 \text{ for all }
    v \in
    V\} = \{0 \}.
  \end{equation}
  To prove~\eqref{eq:adjinj}, consider an $\ell$ satisfying
  $(Av,\ell)_L =0$ for all $v \in V$. Then, for all $w_0\in W_0
  \subseteq V$, we have by~\eqref{eq:11} that $\ip{ \At
    \ell,w_0}_{W_0}=0$, from which it follows, by the density of $W_0$
  in $L$, that $\At \ell =0$ and $\ell \in W.$ Hence we may
  apply~\eqref{eq:D}, which yields
  \begin{equation}
    \label{eq:1}
    \ip{ D v, \ell}_W = (Av, \ell)_L - (v,\At \ell)_L=0, \quad \forall v \in V.
  \end{equation}
  Thus $\ell \in \lprp{D(V)} = V^*$, and so~\eqref{eq:Atbdbelow} implies
  $\ell=0$. This proves~\eqref{eq:adjinj}.

  That $\At$ is a bijection is proved similarly.
\end{proof}

\begin{remark} \label{rem:closed} %
  \rev{Note that under the assumptions of Theorem~\ref{thm:bddbelow},
    if~\eqref{eq:Abdbelow} holds and $\At$ is injective,
    then~\eqref{eq:Atbdbelow} holds with the same constant $c$.  This
    is most easily seen by viewing $A$ as a closed (possibly
    unbounded) operator on $L$ with $\dom(A) = V$.  From~\eqref{eq:3},
    we know that $\D \subset W_0 \subset V \subset L$. Since $\D$ is
    dense in $L$, the domain of $A$ is dense in $L$. Hence the adjoint
    $A'$ is a well defined closed operator on $L$ satisfying
    $(Av,s)_L = (v,A's)_L $ for all $v \in \dom(A)$ and
    $s \in \dom(A')$.  By definition, $\dom(A')$ consists of all
    $s \in L$ for which there exists an $\ell \in L$ with the property
    $(s,A v)_L=(\ell,v)_L$ for all $v \in \dom(A)=V$. In particular,
    whenever $s \in \dom(A')$, by~\eqref{eq:11},
    $(s,A w_0)_L= \ip{ \At s, w_0}_{W_0} = (\ell, w_0)_L$ for all
    $w_0 \in W_0 \subset V$, so $\At s = \ell$ and consequently
    $s \in W$. Thus, whenever $s \in \dom(A')$, both $(s, Av)_L$ and
    $(\At s, v)_L$ coincide with $(\ell,v)_L$ for all $v \in V$, which
    by~\eqref{eq:D}, implies that $\ip{D v, s}_W =0 $, which in turn
    implies that $ s \in V^* = \lprp{D(V)}$ because of~\eqref{eq:V},  i.e.,
    $\dom(A') \subseteq V^*$. Combining with the easily provable reverse
    inclusion, $\dom (A') = V^*$. Next, we claim that
    $A'$ is the same as $\At$: Indeed,
    $ (Av, s)_L - (v,A's)_L = \ip{ D v, s}_W = 0$ for all
    $ v \in \dom(A)=V$ and $s \in \dom(A')=V^*=\lprp{D(V)}$, thus
    proving the claim.  Now, since~\eqref{eq:Abdbelow} implies that
    the $\ran(A)$ is closed, by the Closed Range Theorem for closed
    operators~\cite{Kato95}, we conclude that $\ran(A') = \ran(\At)$
    is closed. Hence if $\At$ is also injective, then by standard
    arguments,~\eqref{eq:Atbdbelow} follows.}

\end{remark}

To summarize, we have discussed two known approaches to abstract Friedrichs
systems and introduced a new sufficient condition for unique
  solvability of Friedrichs problems.  The first approach 
via~\eqref{eq:M} is closer to the classical theory 
(the $M$-approach) while the second is the approach
via~\eqref{eq:V} (the $V$-approach). Whether these two approaches are
equivalent is a natural question. It was shown
in~\cite{ErnGuermCapla07} that if an operator $M$ exists that
satisfies~\eqref{eq:M}, then $V=\ker(D-M)$ and $V^*=\ker(D+M^*)$
satisfies~\eqref{eq:V}. The converse remained unknown until it was
proven in~\cite{AntonBuraz10}. In the remainder of this paper, we will
use only the $V$-approach.

\section{A weak formulation with boundary fluxes}  \label{sec:weak-form}

Consider the following abstract boundary value problem: Given $f\in L$
and $g \in W$, find $u \in W$ satisfying
\begin{subequations}
  \label{eq:bvp-V}
  \begin{align}
    \label{eq:bvp1-V}
    A u & = f, \\
    \label{eq:bvp2-V}
    u-g & \in V.
  \end{align}
\end{subequations}
Space-time Friedrichs systems with non-homogeneous conditions on
space-time boundaries (which includes initial conditions) can be
abstracted into this form.

To derive a  weak formulation, we multiply~\eqref{eq:bvp1-V} by a
test function $v \in W$ and use~\eqref{eq:D}, to obtain
$
(u,\At v)_L + \ip{ D u,v}_W = (f,v)_L.
$
This implies 
\begin{equation}
  \label{eq:2}
(u,\At v)_L + \ip{ D(u-g),v}_W = F(v),
\end{equation}
where
\begin{equation}
  \label{eq:F-V}
  F(v) =(f,v)_L - \ip{D g,v}_W.
\end{equation}
Now, we let $D(u-g)$ in~\eqref{eq:2} be an independent ``flux''
variable $q$. This leads us to formulate the following variational
problem:
\begin{equation}
  \label{eq:uwp-V}
  \begin{aligned}
    \text{\em Find $u\in L$} &\text{ \em and  $q \in \Vsp$ such that}
    \\
    \qquad &(u,\At v)_L + \ip{ q,v }_W = F(v),\qquad \forall v \in W.
  \end{aligned}
\end{equation}
The bilinear form on the left hand side will be denoted by
$b((u,q),v)$. Our approach to the construction and analysis of this
weak formulation is close (but not identical) to the approach 
in~\cite{Bui-TDemkoGhatt13}. 

A similar derivation for the adjoint problem
of finding a $\tilde u \in W$, given $f\in L$ and $g \in W$, such that
\begin{subequations}
  \label{eq:bvp-V-adj}
  \begin{align}
    \label{eq:bvp1-V-adj}
    \At \tilde u & = f, \\
    \label{eq:bvp2-V-adj}
    \tilde u-g & \in V^*,
  \end{align}
\end{subequations}
suggests the following dual weak formulation:
\begin{equation}
  \label{eq:uwp-V-adj}
  \begin{aligned}
    \text{\em Find $\tilde u\in L$} &\text{ \em and $\tilde q \in
      \Vp$ such that}
    \\
    \qquad &(\tilde u, A v)_L - \ip{ \tilde q,v }_W = \tilde F(v),\qquad
    \forall v \in W,
  \end{aligned}
\end{equation}
where
\begin{equation}
  \label{eq:Ft-V}
  \tilde F(v) = (f,v)_L + \ip{ D g,v}_W.
\end{equation}
The bilinear form on the left hand side will now be denoted by $\tilde
b((\tilde u,\tilde q),v)$.

In applications, the $\ip{\cdot,\cdot}_W$ terms can typically be
identified as boundary terms, so $q$ and $\tilde q$ can be interpreted
as boundary fluxes.  Finally, note that by virtue of~\eqref{eq:10}, we
can equivalently use $D(V)$ and $D(V^*)$ as the flux spaces
in~\eqref{eq:uwp-V} and~\eqref{eq:uwp-V-adj}, respectively.

\subsection{Wellposedness}

Next, we prove that the new weak formulation is well posed and is
equivalent to the classical formulation~\eqref{eq:bvp-V} in the
following sense.

\begin{theorem}
  \label{thm:wellposedness-V}
  Suppose~\eqref{eq:V} and~\eqref{eq:A-V} hold.
  Then the following statements hold:
  \begin{enumerate}[(a)]
  \item Given any $F\in W'$, there is a unique $(u,q)\in L \times \Vsp$
    that solves~\eqref{eq:uwp-V}.  Moreover, if $F$ is as
    in~\eqref{eq:F-V} for some given $f\in L$ and $g\in W$, then the
    solution $(u,q)$ of~\eqref{eq:uwp-V} satisfies
    \begin{equation}
      \label{eq:18}
    Au = f,\qquad
    u-g \in V, \qquad
    q = D(u-g).
    \end{equation}

 \item Given any $\tilde F\in W'$, there is a unique $(\tilde u,\tilde
  q)\in L \times \Vp$ that solves~\eqref{eq:uwp-V-adj}.  Moreover,
  if $\tilde F$ is as in~\eqref{eq:Ft-V} for some given $f\in L$ and $g\in W$,
  then the solution $(\tilde u,\tilde q)$ of~\eqref{eq:uwp-V-adj} satisfies
  \[
  \At \tilde u = f,\qquad
  \tilde u-g \in V^*,\qquad
  \tilde q = D(\tilde u - g).
  \]
  \end{enumerate}
\end{theorem}

To prove this theorem, we will verify a uniqueness and an inf-sup
condition in the following lemmas.

\begin{lemma}[Uniqueness]
  \label{lem:uniqueness-V}
  Suppose~\eqref{eq:V} and~\eqref{eq:A1-V} hold. Then, whenever $u\in L$
  and $q\in \Vsp$ satisfies $b((u,q),v)=0$ for all $v \in W$, we have
  $(u,q)=0$.
\end{lemma}
\begin{proof}
  Suppose
  \begin{equation}
    \label{eq:12-V}
    (u,\At v )_L +   \ip{ q,v }_W = 0 \qquad \forall v \in W.
  \end{equation}
  Since $q \in \Vsp$, we have $ \ip{ q, v }_W = 0 $ for all $v \in
  W_0$ due to~\eqref{eq:3}.  Hence, choosing $v=v_0\in W_0$
  in~\eqref{eq:12-V}, we conclude that $(u,\At v_0)_L=0$.  Hence,
  using~\eqref{eq:11}, we have $\ip{ A u, v_0}_{W_0}=0$ for all $v_0
  \in W_0$, which implies, by density, that $Au=0$ in $L$. In
  particular, this shows that $u$ is in $W$.  We may therefore
  apply~\eqref{eq:D} to~\eqref{eq:12-V} to get $ (A u, v)_L - \ip{ D
    u,v}_W +  \ip{q,v }_W = 0, $ for all $v \in W.$
  Since $Au=0$,
  \[
  \ip{D v,u}_W = \ip{ q,v}_W \quad\forall v \in W.
  \]
  Since $q \in \Vsp$, the right hand side vanishes for all $v \in
  V^*$, so $u \in \lprp{ D(V^*)}$. Hence by assumption~\eqref{eq:V},
  $u \in V$.  By~\eqref{eq:A1-V}, $u=0$. Using this
  in~\eqref{eq:12-V}, it also follows that $q=0$.
\end{proof}

\begin{lemma}[Inf-sup condition]
  \label{lem:infsup-V}
  Suppose~\eqref{eq:V} and~\eqref{eq:A1-V} hold. Then, there is a
  $C>0$ such that for all $v\in W$,
  \[
  C \| v \|_W\le
  \sup_{(u,q)\in L\times \Vsp}
  \frac{ | b( (u,q),v) |}{ \| (u,q) \|_{L\times W'} },
  \qquad \forall v \in W.
  \]
\end{lemma}
\begin{proof}
  By~\eqref{eq:A1-V}, there is a $c>0$ such that given any $v \in W$,
  there is a unique $w\in W$ satisfying
  \begin{subequations}
    \label{eq:4}
    \begin{align}
      \label{eq:4-1}
      A w& =v,
      \\
      \label{eq:4-2}
      w& \in V,
      \\
      \label{eq:4-3}
      \| w \|_W & \le c \| v \|_L.
  \end{align}
\end{subequations}
  Then, since $w \in V = \lprp{D(V^*)}$, we have $\ip{ D v^*,w}_W =0$
  for any $v^* \in V^*$. Therefore, $q = Dw$ is in $\Vsp$.
  Moreover,
  \begin{align*}
    \| v \|_L^2 + \| \At v\|_L^2
    & = ( A w, v)_L + (\At v,\At v)_L &&\text{by~\eqref{eq:4-1}}
    \\
    & = (w+\At v, \At v)_L + \ip{ D w,v}_W && \text{by~\eqref{eq:D}}
  \end{align*}
  Note that the right hand side equals $b( (w+\At v,q),v)$ as
  $q = Dw$. Continuing,
  \begin{align*}
    \| v \|_L^2 + \| \At v\|_L^2
    & = \frac{b( (w+\At v,q),v) }{ \| (w + \At v, q)\|_{L\times W'}  }
    \| (w + \At v, q)\|_{L\times W'}
    \\
    & \le \left(
      \sup_{(z,r)\in L\times \Vsp} \frac{ | b( (z,r),v) | }{ \| (z,r)\|_{L\times W'}  }
    \right)
    \| (w + \At v, q)\|_{L\times W'}.    
  \end{align*}

  By Proposition~\ref{prop:collect}(\ref{item:equiv})
  and~\eqref{eq:4-3}, $\| w + \At v \|_L\le \| w \|_L + \| v \|_{\tilde
    W}\le C \| v\|_W $ for a $C>0$ depending on $c$. Moreover, if $d$
  denotes the norm of $D$, then $\|q \|_{W'} \le d\| w \|_W\le c d
  \|v\|_W$. Using these estimates to bound $ \| (w + \At v,
  q)\|_{L\times W'}$, the lemma is proved.
\end{proof}

\begin{proof}[Proof of Theorem~\ref{thm:wellposedness-V}]
  Lemmas~\ref{lem:uniqueness-V} and~\ref{lem:infsup-V} verify the
  conditions of the \Babuska-Brezzi theory, from which the stated
  unique solvability follows.

  Now suppose $F$ is expressed in terms of $f$ and $g$ as
  in~\eqref{eq:F-V}. Then choosing $v=v_0\in W_0$ within the weak
  formulation, 
  \[
  (u,\At v_0 )_L +   \ip{ q,v_0 }_W = (f,v_0)_L -  \ip{ Dg,v_0}_W
  = (f,v_0)_L
  \]
  we obtain $(u,\At v_0)_L=\ip{ Au, v_0}_W = (f,v_0)_L$. This proves,
  by density, that $Au = f$ in $L$, and consequently $u\in W$. Then,
  returning to~\eqref{eq:uwp-V} and using~\eqref{eq:D} together with
  $Au =f$, we obtain $ \ip{ q,v}_W = \ip{ D (u-g),v}_W $ for all
  $v \in W$, i.e., $q = D(u-g)$.  Finally to show that
  $u-g \in V =\lprp{D(V^*)}$, consider an arbitrary $v^* \in V^*$.
  Then note that $q\in \Vsp$, so
  \[ 0=\ip{ q,v^* }_W=\ip{ D (u-g),v^* }_W =\ip{ D v^*, u-g }_W.
  \]
  This proves~\eqref{eq:18}. 
  The remaining statements are proved similarly using~\eqref{eq:A2-V}
  in place of~\eqref{eq:A1-V}.
\end{proof}

\section{Examples}   \label{sec:examples}

The assumptions on which the previous theory is based can be verified
for several examples. We begin with the simplest example in one space
dimension in \S~\ref{ssec:dt} where all the ideas are transparent. We
then generalize to the example of multidimensional advection in
\S~\ref{ssec:advect} and establish a new trace theorem for the
associated graph space. The final example in \S~\ref{ssec:1dgeneral}
considers a general symmetric hyperbolic system in one space dimension
and leads into the discussion on the wave equation in the subsequent
section.

\subsection{An example with no space derivatives}   \label{ssec:dt}

  We begin with a simple example in one space dimension that
  illustrates the essential points. Let $K$ denote the open triangle
  in space-time $(x,t) \in \RRR\times \RRR$, with vertices at $(x,t) =
  (0,0), (1,0),$ and $(1,1)$. Set
  \begin{equation}
    \label{eq:A=dt}
    L = L^2(K), \quad
    \D = \D(K), \quad
    A u = \frac{\d u }{\d t}
  \end{equation}
  (where $\D(K)$ denotes the set of compactly supported infinitely
  differentiable functions on~$K$).  Obviously, $\At = -\d_t$,
  so~\eqref{eq:T} is satisfied. We split the boundary of $K$ into
  an inflow, outflow, and a characteristic part:
  \begin{gather*}
    \din K   = \{ (x,t) \in \d K: t=0\}, \qquad 
    \dout K   = \{ (x,t) \in \d K: x=t\},\\
    \dc K   = \{ (x,t) \in \d K: x=1\}.
  \end{gather*}
  Because $\dist( \din K, \dout K) =0$, although the operator
  $D$ is defined on all $W$, we must be careful in speaking of traces
  of functions in $W$ on these boundary parts.  Indeed, 
$w (x,t) = x^{-1/2}$ is in $W$, but its restriction to
  $\din K$ is not in $L^2(\din K)$.

  To study this further, define the maps
  \[
  \tin : v(x,t) \mapsto v(x,0)
  \quad\text{and}\quad
  \tout : v(x,t) \mapsto v(x,x),
  \]
  whose application to any function gives its traces on $\din K$ and
  $\dout K$, respectively. These maps are obviously well defined for
  smooth functions. Below we prove that they extend to $W$. Let
  $L_w^2(S)$ denote the set of all measurable functions $s$ on $S$
  with finite $\int_S w s^2$.

  \begin{lemma}
    \label{lem:trace-dt}
    For the $W$ in this example, the following maps are
    continuous:
    \[
    \tin : W \to L_x^2(0,1), \quad
    \tout : W \to L_x^2(0,1),
    \quad\text{and}\quad
    \tin - \tout : W \to L_{1/x}^2(0,1),
    \]
    i.e., there is a constant $C_0>0$ such that
    \begin{equation}
      \label{eq:5}
      \int_0^1 x |\tin w |^2 \; dx
      +
      \int_0^1 x |\tout w |^2 \; dx
      +
      \int_0^1 \frac {|\tin w - \tout w |^2}{x} \; dx
      \le C_0 \| w \|_W^2
    \end{equation}
    for all $w \in W$.
  \end{lemma}
  \begin{proof}
    A general density result in~\cite[Theorem~4]{Jense04} implies that
    $C^1(\bar K)$ is dense in $W$,
    so it suffices to prove~\eqref{eq:5} for all $w \in C^1(\bar K)$.
    Beginning with the fundamental theorem of calculus,
    \begin{align*}
      \tin w(x) = w(x,r) - \int_0^r \d_t w(x,s)\; ds,
    \end{align*}
    squaring, integrating over $r$, and overestimating,
    \begin{align*}
      x |\tin w(x)|^2
      & =
      \int_0^x |\tin w(x)|^2 \; dr
      \le 2 \int_0^x | w(x,r)|^2\; dr
       + 2 \int_0^x \! r\! \int_0^r |\d_t w(x,s)|^2 \; ds \, dr.
    \end{align*}
    Now integrating over $x$ and overestimating again,
    \begin{align*}
      \frac  1 2 \int_0^1 x |\tin w(x)|^2 \; dx
      & \le
        \int_0^1 \!\!\! \int_0^x | w(x,r)|^2\; dr\, dx
       +  \int_0^1\!\!\! \int_0^1 \! 1\! \int_0^x |\d_t w(x,s)|^2 \; ds \, dr\, dx
       \\
       & =  \| w \|_W^2.
    \end{align*}
    A similar argument shows that the same inequality holds with
    $\tin$ replaced by $\tout$.

    To complete the proof, we therefore only need to show that
    \begin{equation}
      \label{eq:6}
      \int_0^1 \frac {|\tin w - \tout w |^2}{x} \; dx
      \le \| w \|_W^2.
    \end{equation}
    But this follows from
    \[
    |\tin w(x) - \tout w(x)|^2 =
    \left|\int_0^x \d_t w(x,s) \; ds\right|^2
    \le
    x\int_0^x |\d_t w(x,s)|^2 \; ds,
    \]
    dividing through by $x$ and integrating over $x$.
  \end{proof}

  \begin{lemma}
    \label{lem:V-for-dt}
    Assumption~\eqref{eq:V} holds for this example after
    setting
    \begin{subequations}
      \label{eq:VV1*}
    \begin{align}
      V & = \{ w \in W: \tin w =0 \},
      \\
      V^* & = \{ w \in W: \tout w =0 \}.
    \end{align}
    \end{subequations}
  \end{lemma}
  \begin{proof}
    For $v,w \in C^1(\bar K)$, the definition of $D$ implies that
    \begin{align} \nonumber 
    \ip{ D w,v}_W
    & = \int_0^1\!\!\! \int_0^x (\d_t w) v + w (\d_t v) \; dt \,dx
      \\  \label{eq:9}
      & = \int_0^1 (\tout w)(\tout v)\; dx - \int_0^1 (\tin w) (\tin v) \; dx.
    \end{align}
    In order to apply the density argument, we rewrite this expression:
    \begin{align}
       \label{eq:7}
       \ip{ D w,v}_W
       & = \int_0^1 (x^{1/2}\tout w)
       \left( \frac{\tout v - \tin v}{x^{1/2}} \right) \; dx
       +
       \int_0^1 \left( \frac{\tout w 
         - \tin w}{x^{1/2}} \right) (x^{1/2}\tin v) \; dx.
    \end{align}
    Now, one can immediately verify using Cauchy-Schwarz inequality and
    Lemma~\ref{lem:trace-dt}, that both the integrals extend
    continuously to $W$. Hence~\eqref{eq:7} holds for all $v$ and $w$
    in $W$.  Similarly, the expression
    \begin{align}
      \label{eq:8}
      \ip{ D w,v}_W
      & = \int_0^1
      \left( \frac{\tout w - \tin w}{x^{1/2}} \right) (x^{1/2}\tout v) \; dx
      +
      \int_0^1 (x^{1/2}\tin w) \left( \frac{\tout v - \tin v}{x^{1/2}} \right) \; dx,
    \end{align}
    also holds for all $v$ and $w$ in $W$.

    Let us verify~\eqref{eq:V1}. For any $v \in V$, since $\tin v =0$,
    we have from~\eqref{eq:8} that
    \[
    \ip{ D v,v}_W = \int_0^1 \left( \frac{\tout v - 0}{x^{1/2}}
    \right) (x^{1/2}\tout v) \; dx \ge 0.
    \]
    Hence $V \subseteq C^+$. Similarly, $V^* \subseteq C^-$.

    To prove~\eqref{eq:V2}, let $v \in V$. Then using~\eqref{eq:7} and
    putting $\tin v=0$, we have,
    \[
    \ip{D v^*,v}_W =
    \int_0^1
    (x^{1/2} \tout v^*)
    \left( \frac{\tout v - 0}{x^{1/2}} \right) \; dx
    \]
    which vanishes for any $v^* \in V^*$. Hence $V \subseteq
    \lprp{D(V^*)}$. For the reverse inclusion, let $v^\perp \in
    \lprp{D(V^*)}$. Then, since $\tout v^* =0$ for all $v^* \in V^*$,
    we have from~\eqref{eq:7} that
    \[
    \ip{D v^*,v^\perp}_W =
    \int_0^1 \left( \frac{\tout v^* - \tin v^*}{x^{1/2}} \right) (x^{1/2}\tin v^\perp)
    \; dx
    =
    -\int_0^1  (\tin v^*) (\tin v^\perp)\; dx.
    \]
    Since all functions in $\D(0,1)$ can be written as $\tin v^*$ for
    some $v^* \in V^*$, this implies that $\tin v^\perp =0$ a.e.\ in
    $(0,1)$, so $v^\perp \in V$. Thus, $ V = \lprp{D(V^*)}$. A similar
    argument shows that $ V^* = \lprp{D(V)}.$
  \end{proof}

  \begin{remark}
    Note that although the two integrals in~\eqref{eq:9} need not
    generally exist for all $w,v \in W$, those in the
    identities~\eqref{eq:7} and \eqref{eq:8} exist for all
    $w,v \in W$.
  \end{remark}

  \begin{remark}
    It is proved in~\cite[Lemma~4.4]{ErnGuermCapla07} that if $V+V^*$
    is closed, then an $M$ that satisfies~\eqref{eq:M} can be
    constructed. They then write, ``it is not yet clear to us whether
    properties~\eqref{eq:V1}--\eqref{eq:V2} actually imply that $V +
    V^*$ is closed in $W$.'' This issue was settled
    in~\cite{AntonBuraz10} where they showed by a counterexample that
    \eqref{eq:V1}--\eqref{eq:V2} does not in general imply $V + V^*$
    is closed.  Our study above provides another simpler counterexample:
    Specifically, for $n \ge 2$, let $\chi_n$ denote the indicator
    function of the interval~$[1/n,1]$.  Then $v_n(x,t) = \chi_n(x)
    t/x$ is in $V$ and $v_n^* = \chi_n(x) (x-t)/x$ is in
    $V^*$. Clearly, as $n\to \infty$, the sequence $v_n + v_n^* =
    \chi_n \in V+V^*$ converges in $W$.  But its limit, the function
    $1$, is not in $V + V^*$. Indeed, if $1$ were to equal $v + v^*$
    for some $v\in V$ and $v^*\in V^*$, then by
    Lemma~\ref{lem:trace-dt},
    \[
    \int_0^1 \frac{1}{x}\; dx =
    \int_0^1 \frac{|\tin v + \tout v^*|^2}{x}\; dx
    \le
    2\int_0^1 \frac{|\tin v|^2}{x} + \frac{|\tout v^*|^2}{x} \; dx
    \le 2C \left( \| v \|_W^2+   \| v^* \|_W^2\right)
    \]
    which is impossible.
  \end{remark}

  \begin{lemma}
    \label{lem:bddbelow}
    The inequalities of~\eqref{eq:AtAbelow} hold for this example.
  \end{lemma}
  \begin{proof}
    Given any $v \in V$, by the density of $C^1(\bar K)$ in $W$, there
    is a sequence $w_n \in C^1(\bar K)$ converging to $v$ in $W$.  Let
    $v_n (x,t) = w_n(x,t) - (\tin w_n)(x)$. Clearly
    $v_n \in V \cap C^1(\bar K).$ Moreover, 
    \begin{align*}
      \| v_n - v\|_W 
      &
     \le 
        \| w_n - v \|_W + \| \tin w_n \|_{L^2(K)}
      \\
      & 
        = \| w_n - v \|_W + 
        \left( \int_0^1 \!\! \int_0^x |w_n(x,0)|^2 \; dt\, dx\right)^{1/2}
      \\
      & 
        = \| w_n - v \|_W +  \| \tin w_n \|_{L^2_x(0,1)}
      \\
     & = \| w_n - v \|_W +  \| \tin (w_n - v) \|_{L^2_x(0,1)}.
    \end{align*}
    This, together with Lemma~\ref{lem:trace-dt}, and the convergence
    of $w_n$ to $v$ in $W$, imply the convergence of $v_n$ to $v$ in
    $W$. Thus $V \cap C^1(\bar K)$ is dense in $V$.  Similarly,
    $V^* \cap C^1(\bar K)$ is dense in $V^*$.  Hence, it suffices to
    prove the inequalities of~\eqref{eq:AtAbelow} for the dense
    subsets.

    For any $C^1(\bar K)$ function $v$ in $V$,  we have
    \begin{align*}
      v(x,t)^2
      & = \int_0^t \frac{ \d }{\d s}\left(v(x,s)^2\right)\; ds
      = \int_0^t 2v(x,s) \frac{ \d }{\d s}v(x,s)\; ds,
    \end{align*}
    which implies
    \begin{align*}
      \int_0^1\!\!\!\int_0^x v(x,t)^2 \; dt\,dx
      &
      \le
      2\int_0^1\!\!\!\int_0^x
      \left( \int_0^t v(x,s)^2\;ds\right)^{1/2}
      \left( \int_0^t \left|\frac{ \d }{\d s}v(x,s)\right|^2\; ds\right)^{1/2}dt\,dx,
      \\
      & \le
      2\int_0^1
      \left( \int_0^x \!\!\!\int_0^t v(x,s)^2\;ds\,dt\right)^{1/2}
      \left( \int_0^x\!\!\!\int_0^t \left|\frac{ \d }{\d s}v(x,s)\right|^2
        \; ds\,dt\right)^{1/2}dx.
    \end{align*}
    Since $x \le 1$, this shows that $\| v \|_L \le 2 \| A v \|_L$ for
    all $v \in V \cap C^1(\bar K)$ and hence for all $v \in V$. The
    proof of~\eqref{eq:Atbdbelow} is similar.
  \end{proof}

  \begin{theorem}
    \label{thm:eg-dt}
    Formulations~\eqref{eq:uwp-V} and~\eqref{eq:uwp-V-adj} are well posed
    for this example.
  \end{theorem}
  \begin{proof}
    By Lemmas~\ref{lem:V-for-dt} and~\ref{lem:bddbelow},
    assumptions~\eqref{eq:V} and~\eqref{eq:AtAbelow} hold, so
    Theorem~\ref{thm:bddbelow} implies that~\eqref{eq:A-V} holds.
    Therefore, Theorem~\ref{thm:wellposedness-V} gives the required
    result.
  \end{proof}

  \subsection{Unidirectional advection}   \label{ssec:advect}

  The above calculations have a straightforward generalization to
  multidimensional tent-shaped domains. We say that $K_0$ is a vertex
  patch around a point $p$ if it is an open polyhedron in $\RRR^d$ ($d\ge 1$)
  that can be partitioned into a finite number of $d$-simplices with a
  common vertex~$p\in \RRR^d$.  

  We first consider domains $K$ built on (spatial) vertex patches of
  the form
  \begin{equation}
    \label{eq:K}
    K = \{ (x,t): \; x\in K_0, \; \gi(x) < t < \go(x) \}
  \end{equation}
  (and later, after Definition~\ref{defn:tent} below, specialize to
  tent-shaped domains). Above, $\go(x)$ and $\gi(x)$ are Lipschitz
  functions on $K_0$ such that $K$ is a nonempty open set in
  $\RRR^{d+1}$.  Then the unit outward normal vector $n=(n_x,n_t)$
  exists a.e.~on $\d K$.  Continuing to consider the same operator as
  in~\eqref{eq:A=dt}, namely $A = \d_t$, but on the new domain $K$,
  the following defines inflow, outflow, and characteristic parts of
  the boundary:
  \begin{subequations}
    \label{eq:dtoutinbdry}
    \begin{gather}
    \din K   = \{ (x,t) \in \d K:  n_t<0\},
    \qquad
    \dout K   = \{ (x,t) \in \d K: n_t>0\},
      \\
      \dc K   = \{ (x,t) \in \d K: n_t=0 \}.
    \end{gather}
  \end{subequations}
  We can immediately prove the following by extending the arguments of
  \S\ref{ssec:dt}.

  \begin{theorem}
    \label{thm:multidim_dt}
    Let $K$ be as in~\eqref{eq:K} and let $A = \d_t$. Then the inflow
    and outflow trace maps, $ \tin : v(x,t) \mapsto v(x,\gi(x))$ and
    $\tout : v(x,t) \mapsto v(x,\go(x))$, extend to continuous linear
    operators
    \[
    \tin : W \to  L_{\go-\gi}^2(K_0), \quad
    \tout : W \to L_{\go-\gi}^2(K_0),
    \quad\text{and}\quad
    \tin - \tout : W \to L_{1/(\go-\gi)}^2(K_0).
    \]
    As in~\eqref{eq:VV1*}, set $V = \ker(\tin)$ and
    $V^* = \ker(\tout)$. Then, the assumptions of~\eqref{eq:V} and the
    inequalities of~\eqref{eq:AtAbelow} hold.  Hence the
    formulations~\eqref{eq:uwp-V} and~\eqref{eq:uwp-V-adj} are
    well-posed.
  \end{theorem}

  Identities similar to~\eqref{eq:7} and~\eqref{eq:8} prove the
  continuity properties of the trace maps stated above. To prove the
  stated wellposedness, we need to verify the assumptions
  in~\eqref{eq:V} and~\eqref{eq:AtAbelow}, which can be done by simple
  generalizations of the arguments in the proofs of
  Lemmas~\ref{lem:V-for-dt} and~\ref{lem:bddbelow}.  Next, we proceed
  to consider a convection operator on tent-shaped domains.

  \begin{definition} \label{defn:tent} Suppose $K$ and $K_0$ are as
    in~\eqref{eq:K}.  If, in addition, $K$ can be divided into
    finitely many $(d+1)$-simplices with a common edge
    $\{ (p, t): \gi(p) < t < \go(p)\}$, then we call $K$ a space-time
    {\em tent}. We refer to the common edge as its {\em tent pole}.
    Clearly, in this case, $\go$ and $\gi$ are linear on each simplex
    of $K_0$.  We split the tent's boundary into the these parts:
\begin{subequations}
  \label{eq:macbdrysplit}
\begin{gather}
  \label{eq:macbdrysplitio}
  \din K  = \{ (x, \gi(x)) : x \in K_0 \}, \qquad
  \dout K  = \{ (x, \go(x)) : x \in K_0 \},
  \\
  \db K  = \d K \setminus ( \din K \cup \dout K).
\end{gather}
\end{subequations}
We refer to the two parts in~\eqref{eq:macbdrysplitio} as the tent's
{\em inflow} and {\em outflow} boundaries, respectively. (Using such
terms without regard to an underlying flow operator is an abuse of
terminology that we overlook for expediency.)
  \end{definition}

  The equation modeling advection along a fixed direction
  $\alpha \equiv (\alpha_i) \in \RRR^d$ is of the form $A u = f$ with
  \begin{equation}
    \label{eq:17}
    A u = \frac{\d u }{\d t} + \sum_{i=1}^d \alpha_i  \frac {\d u }{\d x_i}.
  \end{equation}
  Setting $L = L^2(K)$, $\D = \D(K)$, and noting that $\At = -A$, we
  can put this into the Friedrichs framework since the
  prerequisite~\eqref{eq:T} holds.

Let $n \in \RRR^{d+1}$ denote the outward unit normal on $\d K$. We
often write it separating its space and time components as $n = ( n_x
, n_t)$ with $n_x \in \RRR^d$. We now assume that the tent
boundaries are such that
\begin{subequations}
  \label{asm:inoutassume}
\begin{align}
  \din K & \subseteq \{ (x,t) \in \d K: n \text{ at } (x,t) 
           \text{ satisfies } n_t + \alpha \cdot n_x <0 \} \\
  \dout K & \subseteq \{ (x,t) \in \d K: n \text{ at } (x,t) \text{ satisfies } n_t + \alpha \cdot n_x > 0 \}.
\end{align}
\end{subequations}
The vertical part of the boundary, namely $\db K$, is further split
into three parts $\db^+ K, \db^- K,$ and $\db^0 K$ where
$n_t + \alpha \cdot n_x = \alpha \cdot n_x$ is positive, negative, and
zero, respectively (see Figure~\ref{fig:proof}).  Let $\Gin$ and
$\Gout$ denote the closures of $\din K \cup \db^- K$ and
$\dout K \cup \db^+ K$, respectively, and let
\[
\Gio = \Gin \cap \Gout.
\]
Define $\delta(z) = \dist(z, \Gio)$. We will use the restriction of this
function to $\Gin$ and $\Gout$ as weight functions while describing
the norm continuity of traces below.  For smooth functions $w$ on $K$,
let
\[
\tin w = w|_{\Gin},
\qquad
\tout w = w|_{\Gout}.
\]

\begin{theorem}
  \label{thm:advect}
  Let $K$ be a tent and $A$ be given
  by~\eqref{eq:17}. Suppose~\eqref{asm:inoutassume} holds. Then the
  above-defined maps $\tin$ and $\tout$ extend to continuous linear
  operators
  \[
  \tin : W \to L_{\delta}^2(\Gin) \quad\text{ and }\quad
  \tout : W \to L_{\delta}^2(\Gout).
  \]
  Hence $V=\ker(\tin)$ and $V^*=\ker(\tout)$ are closed subspaces of
  $W$.  When restricted to these subspaces, the traces have an
  additional continuity property, namely
  \begin{equation}
    \label{eq:29}
  \tin : V^* \to L^2_{1/\delta}(\Gin) \quad\text{ and }\quad
  \tout : V \to L^2_{1/\delta} (\Gout)    
  \end{equation}
  are continuous.  Finally, with this $V$ and $V^*$, the weak
  formulations~\eqref{eq:uwp-V} and~\eqref{eq:uwp-V-adj} are
  well-posed.
\end{theorem}
\begin{proof}
  The idea is to use a change of variable that brings the operator to
  the previously analyzed operator $\d_t$. The new variables are
  $\hx = x - \alpha t$ and $ \htt = t$, i.e.,
  \[
  \begin{bmatrix}
    x \\ t
  \end{bmatrix}
  = H
  \begin{bmatrix}
    \hx  \\ \htt
  \end{bmatrix}
  \quad \text{ where } \quad
  H =
  \begin{bmatrix}
    I &  \alpha \\
    0 & 1
  \end{bmatrix}.
  \]
  Let $\hK = H^{-1} K$. (Note that $\hat K$ is not a tent, in
  general.) Pulling back functions $w$ on $K$ to functions
  $\hw = w \circ H$ on $\hK$, the chain rule gives
  \begin{equation}
    \label{eq:24}
  \hat A \hat w = (A w) \circ H, 
  \quad \text{ where } \quad 
  \hat A = \frac{ \d  }{\d \htt}.
  \end{equation}
  Thus $w\in W$ if and only if $\hw \in \hat W = \{ \hz \in L^2(\hK): 
  \hat A
  \hz \in L^2(\hK)\}$.

  Next, let $\hn = (\hn_{\hat x}, \hn_{\htt}) $ denote the unit
  outward normal on $\d\hat K$. Then
  $\hn = (\hn_\hx,\hn_\htt) = H^t n / \| H^t n\|_2$.  Defining
  $\din \hK, \dout \hK$, and $\dc \hK$ as in~\eqref{eq:dtoutinbdry},
  we claim that
  \begin{subequations}
    \label{eq:19}
  \begin{align}
    \din \hK  
    & \equiv \{ (\hat x,\hat t) \in \d\hat K : 
        \hat n_{\hat t} < 0\} =  H^{-1}( \din K \cup \db^- K) 
    \\
    \dout \hK  
    &\equiv  \{ (\hat x,\hat t) \in \d\hat K : 
      \hat n_{\hat t} > 0\} 
      = H^{-1}( \dout K \cup \db^+ K),
    \\
    \dc \hK
    & \equiv  \{ (\hat x,\hat t) \in \d\hat K : 
      \hat n_{\hat t} = 0\}  = H^{-1}( \db^0 K).
  \end{align}
  \end{subequations}
  For example, to sketch a proof 
  of the first identity, note that $n$ at $(x,\gi(x))$ is in the
  direction of $(\nabla_x \gi, -1)$ where $\nabla_x$ denotes the
  gradient with respect to $x$.  Hence, because
  of~\eqref{asm:inoutassume}, we have
  $\alpha \cdot \nabla_x \gi -1 <0$ on $\din \hat K$.  Since the
  mapped normal $\hn$ is in the direction of
  \[
  H^t n =
  \begin{bmatrix}
    I  & 0 \\
    \alpha^t  & 1
  \end{bmatrix}
  \begin{bmatrix}
    \nabla_x \gi \\ -1
  \end{bmatrix}
  =
  \begin{bmatrix}
    \nabla_x \gi \\
    \alpha \cdot \nabla_x \gi -1
  \end{bmatrix}
  \]
  we conclude that $\hn_{\htt} <0$. Applying similar arguments on the
  remaining parts of the boundary, the claim~\eqref{eq:19} is proved.

Let $\hat K_0$ be the projection of $\hat K$ on the $\hat t=0$ plane.
There are (continuous piecewise linear) functions $\hgo$ and $\hgi$
such that $\dout \hat K$ and $\din \hat K$ are graphs of $\hgo$ and
$\hgi$, respectively, over $\hat K_0$.  On $\hat K$, since
$\hat A = \d_{\hat t}$, we apply Theorem~\ref{thm:multidim_dt} to
conclude that $\htin \hat w = \hat w|_{\din \hat K}$ and
$\htout \hat w = \hat w|_{\dout \hat K}$ extend to continous linear
operators $ \hat \tin : \hat W \to L^2_{\hgo - \hgi}(\hat K_0)$ and
$\htout : \hat W \to L^2_{\hgo - \hgi}(\hat K_0).$ Hence
$\hat V = \ker (\htin)$ and $\hat V^* = \ker (\htout)$ are closed
subspace of $\hat W$.  By the additional continuity of
$\htin - \htout : \hat W \to L^2_{1/(\hgo - \hgi)}(\hat K_0)$ (also
given by Theorem~\ref{thm:multidim_dt}), we conclude that 
\[
\htin : \hat V^* \to L^2_{1/(\hgo - \hgi) }(\hat K_0),
\qquad
\htout : \hat V \to L^2_{1/(\hgo - \hgi) }(\hat K_0),
\]
are also continuous.

\begin{figure}
  \centering
  \begin{tikzpicture}
    \coordinate (top) at  (0,4);
    \coordinate (bot) at  (0,0);
    \coordinate (rgt) at  (2.2,1.8);
    \draw [blue,  thick] (top) -- (bot) node [midway,left] {$\db^- K$};
    \draw [olive, thick] (bot) -- (rgt) node [midway,right] {$\din K$};
    \draw [red,   thick] (rgt) -- (top) node [midway,right] {$\dout K$};
    \draw[->,dotted] (bot) -- ($(bot)+(2,0)$) node[right] {$x$};
    \draw[->,dotted] (top) -- ($(top)+(0,1)$) node[above] {$t$};

    \foreach \a in {0.5}    
    { 
      \path (top); \pgfgetlastxy{\XCoord}{\TCoord};    
      \coordinate (topmap) at ($(\XCoord-\a*\TCoord,\TCoord) + (8,0)$);
      \path (rgt); \pgfgetlastxy{\XCoord}{\TCoord};    
      \coordinate (rgtmap) at ($(\XCoord-\a*\TCoord,\TCoord) + (8,0)$);
      \coordinate (botmap) at ($(bot)                        + (8,0)$);
    }
    \draw [blue,  thick] (topmap) -- (botmap) 
                         node[near end,left] {$\din \hat K$}; 
    \draw [olive, thick] (botmap) -- (rgtmap) 
                         node[midway,right]{$\din \hat K$};
    \draw [red, thick]   (rgtmap) -- (topmap) 
                         node(outflow)[near start,above]{\quad$\dout \hat K$};
    \draw[->,dotted] (botmap) -- ($(botmap)+(2,0)$) node[right] {$\hat x$};
    \draw[->,dotted] (botmap) -- ($(botmap)+(0,5)$) node [above] {$\hat t$};
    
    \coordinate (K) at ($(top)+(1,1.6)$);
    \coordinate (Khat) at ($(K)+(5,0)$);
    \node[below] at (K) {$K$};
    \node[below] at (Khat) {\quad$\hat K$};
    \draw[->]  (K) arc (120:60:5) node[midway,below] {$H^{-1}$} ;
    
    \fill (top) circle (2pt) node[above,left] {$\Gio$};
    \fill (rgt) circle (2pt) node[above,right] {$\Gio$};

    \fill (topmap) circle (2pt);
    \fill (rgtmap) circle (2pt);

    \node[left] at (topmap) {$P$};
    \coordinate (N) at ($(topmap)!0.5!(botmap)$);
    \path (N); \pgfgetlastxy{\XCoord}{\TCoord};    
    \coordinate (zii) at ($(N)+(0,2)$);
    \coordinate (O) at (intersection of N--zii and topmap--rgtmap);
    \draw[dashed,-] (N) node[right]  {$N$} -- (O) node[right] {\;$O$}; 
    \draw (N) circle (1pt);
    \draw (O) circle (1pt);

    \coordinate (arc0) at ($(topmap)!0.3!(N)$);
    \node (cir) [circle through=(arc0)] at (topmap) {};
    \coordinate (arc1) at (intersection of cir and topmap--rgtmap);
    \coordinate (topright) at ($(topmap)+(3,0)$);
    \coordinate (arc2) at (intersection of cir and topmap--topright);
    
    \tkzDrawArc(topmap,arc0)(arc1);
    \tkzDrawArc(topmap,arc1)(arc2);
    \node at ($(arc0)!0.5!(arc1)+0.15*(0.8,-1)$) {$\theta$};
    \draw[dashed] (topmap)-- ++(1,0);
    \node at ($(arc1)!0.5!(arc2)+(0.23,-0.1)$) {$\thetao$};

  \end{tikzpicture}
  \caption{On the left is a tent $K$ with $A = \d_t + 0.5 \d_x$ that
    satisfies~\eqref{asm:inoutassume}. On the right is $\hat K$
    obtained after applying the map in the proof of
    Theorem~\ref{thm:advect} with mapped over operator
    $\hat A = \d_{\hat t}$.}
  \label{fig:proof}
\end{figure}
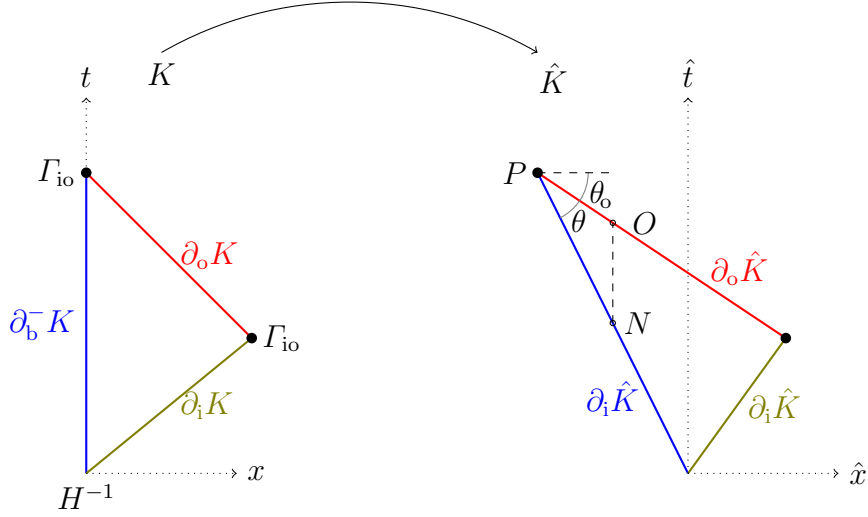

These continuity results are more conveniently mapped to $K$ by using
$\hat \delta(z) = \dist( z, \hGio)$. Note that $\hgo - \hgi$ vanishes
at $\hGio = H^{-1} \Gio$. To restate the continuity properties of
$\tin$ in terms of $\hat \delta$, we prove that there are $c_1, c_2>0$
such that
  \begin{equation}
    \label{eq:22}
    c_1\, \hat \delta(\hat x, \hgi( \hat x) ) \le \hgo(\hat x) - \hgi(\hat x)
    \le c_2 \,\hat \delta(\hat x, \hgi( \hat  x) ),
    \qquad\forall \hat x \in \hat K_0,
  \end{equation}
  (and similarly for $\tout$). When a point
  $N =(\hat x, \hgi(\hat x) )$ on $\din \hat K$ is sufficiently near
  to $\hGio$, the point $P$ nearest to it on $\hGio$, together with
  $O = (\hat x, \hgo(\hat x))$ form a triangle (as shown in
  Figure~\ref{fig:proof}).  Now we may restrict ourselves to the
  two-dimensional plane containing this triangle.

  Consider the case when the segment $PO$ lies on or below the plane
  of constant $\hat t$ passing through $P$, so that $PO$ makes an
  angle $\thetao\ge 0$ with that plane. Let $\theta$ be the angle
  made by $PN$ and $PO$ at $P$. Then, by elementary geometry,
  \begin{equation}
    \label{eq:23}
    \|P-N\|_2 = \frac{\cos \thetao }{\sin \theta} \| O-N \|_2.
  \end{equation}
  Note that $\theta>0$ and $0 \le \thetao<\pi/2$.  Therefore,
  observing that $\| P - N\|_2 = \hat \delta(N)$ and
  $\| O - N \|_2 = \hgo - \hgi$, \eqref{eq:23}
  proves~\eqref{eq:22}. For the remaining geometrical configurations,
  identities similar to~\eqref{eq:23} can be derived to
  prove~\eqref{eq:22}. Having established~\eqref{eq:22}, we find that after
  mapping back to $K$, the stated continuity properties of
  $\tin w = (\htin \hw) \circ H^{-1}$ and
  $\tout w = (\htout \hw)\circ H^{-1}$ are proved.

  It now only remains to prove the stated wellposedness of the weak
  formulations. By Theorem~\ref{thm:multidim_dt}, for any
  $\hat F \in \hat W'$, there is a unique $\hat u \in L^2(\hat K)$ and
  $\hat q = \hat D \hat z \in (\hat V^*)^\perp$ satisfying
  \begin{equation}
    \label{eq:25}
  -(\hat u, \hat A \hat v)_{L^2(\hat K) }
    + \ip{ \hat D \hat z, \hat v }_{\hat W} = \hat F(\hat v),\qquad
    \forall\; \hat v \in\hat W,
  \end{equation}
  where $\hat D \in \L(\hat W, \hat W')$ is defined as before by
  $\ip{ \hat D \hat v, \hat w}_{\hat W} = (\hat A \hat v,\hat
  w)_{L^2(\hat K)} + (\hat v, \hat A \hat v)_{L^2(\hat K)}.$
  Here we have used~\eqref{eq:10} to find a $\hat z \in \hat V$ such
  that $\hat q = \hat D \hat z$. (While $\hat q$ is unique, $\hat z$
  need not be unique.) It now follows from the properties of the
  mapping
that $\hat u$ and
  $\hat q =\hat D \hat z$ satisfies~\eqref{eq:25} if and only if
  $u = \hat u \circ H^{-1}$ and $z =\hat z \circ H^{-1}$ satisfies
  \begin{equation}
    \label{eq:26}
  -(u,  A  v)_{L^2( K) }
    + \ip{  D z,  v }_{ W} = \hat F(v\circ H),\qquad
    \forall\;  v \in W.
  \end{equation}
  Here we have used the fact that~\eqref{eq:24} implies
  $ (\hat u, \hat A \hat v)_{L^2(\hat K) } = ( u, A v)_{L^2( K)}$ and
  consequently,
  $\ip{ \hat D \hat z, \hat v}_{\hat W} = \ip{ D z, v}_{W}$. This
  shows that the weak formulations on $\hat K$ and $K$ are equivalent,
  so the wellposedness of the latter, namely~\eqref{eq:uwp-V}, follows
  from the former.  The wellposedness of~\eqref{eq:uwp-V-adj} is
  proved similarly.
\end{proof}

\begin{remark}
  Under additional assumptions, including $\dist( \din K, \dout K)>0$,
  a stronger trace result is proved
  in~\cite[Lemma~5.1]{ErnGuermCapla07}. However, on tents,
  $\dist( \din K, \dout K)$ is always zero, so we are unable to use
  their result.
\end{remark}

\subsection{A linear symmetric hyperbolic system in one space dimension}  \label{ssec:1dgeneral}

Let $C \in \RRRmm$ be an $m \times m$ real symmetric matrix and let
$K$ be a tent as in~\S\ref{ssec:advect}. Set $L = L^2(K)^m, \D
= \D(K)^m$ and
\begin{equation}
  \label{eq:AC}
  Au = \frac{ \d u }{\d t } + C \frac{ \d u}{\d x}
\end{equation}
where $\d_t u$ and $\d_x u$ are vectors in $\RRR^m$ with their $\ell$th
component equal to $\d_t u_\ell$ and $\d_x u_\ell$, respectively.  Since $C$
is symmetric, $\At = -A$, so assumption~\eqref{eq:T} is obviously
satisfied.

Let $Q$ be an orthogonal matrix and $\vL = \diag(\lambda_\ell)$ be a
diagonal matrix such that $C = Q \vL Q^t$.
Let $\din K, \dout K$ and $\db K$ be as defined
in~\eqref{eq:macbdrysplit}.
In this subsection, we assume -- instead
of~\eqref{asm:inoutassume} -- that
\begin{subequations}
  \label{asm:inoutassumesystem}
\begin{align}
  \din K & \subseteq \{ x \in \d K:  n_t I   + n_x C \text{ is negative definite}\},
  \\
  \dout K & \subseteq \{ x \in \d K: n_t I  + n_x C \text{ is positive definite}\}.
\end{align}
\end{subequations}
For each $\ell =1,\ldots, m$, we decompose $\db K$ into
$\db^{+,\ell} K, \db^{-,\ell} K,$ and $\db^{0,\ell} K$ where
$\lambda_\ell n_x$ is positive, negative, and zero, respectively.  
Then we have the following theorem, which is proved using the
diagonalization of $C$ to separate each component and then appealing
to the analysis in \S~\ref{ssec:advect}. We now opt for a brief
statement of the theorem, leaving the tacitly used properties of the
traces to the proof.

\begin{theorem} \label{thm:hyperbol}
  Suppose~\eqref{asm:inoutassumesystem} holds for the tent $K$ and the
  operator $A$ in~\eqref{eq:AC}.  Then, the
  formulations~\eqref{eq:uwp-V} and~\eqref{eq:uwp-V-adj} with
  \begin{subequations}
    \label{eq:VV*}
  \begin{align}
    V & = \{ z \in W:  [Q^t z]_\ell \big|_{ \din K \cup \db^{-,\ell} K} = 0,
    \text{ for all } \ell=1,\ldots, m\},
    \\
    V^* & = \{ z \in W:  [Q^t z]_\ell \big|_{ \dout K \cup \db^{+,\ell} K} = 0,
    \text{ for all } \ell=1,\ldots, m\},
  \end{align}
  \end{subequations}
  are well-posed.
\end{theorem}
\begin{proof}
  Let $\bA = Q^t A Q,$
  $\bW = \{ \bv \in L^2(K)^m: \bA \bv \in L^2(K)^m\},$ and $\bD$ be
  the corresponding boundary operator on $\bW$. Then clearly,
  $v \in W$ if and only if $\bv = Q^t v$ is in $\bW$.  Moreover,
  $\bA \bw = \d_t \bw + \vL \d_x \bw$, i.e., its $\ell$th component
  equals
  \[
  \bA_\ell \bw_\ell \equiv
  \d_t \bw_\ell + \lambda_\ell \d_x\bw_\ell.
  \]
  Note that $\bA_\ell$ is a Friedrichs operator on $K$ of the
  form~\eqref{eq:17} and has its associated graph space $\bW_\ell$ and
  boundary operator $\bD_\ell$.

  Now, the assumptions of~\eqref{asm:inoutassumesystem} imply that
  \eqref{asm:inoutassume} holds for each $\bA_\ell$ (with
  $\alpha = \lambda_\ell$) so Theorem~\ref{thm:advect} yields the
  continuity of the maps
  $ \btin^\ell: \bw \mapsto \bw_\ell|_{\Gin^\ell} $ and
  $ \btout^\ell: \bw \mapsto \bw_\ell|_{\Gout^\ell}$ on~$\bW$, where
  $\Gin^\ell = \din K \cup \db^{-,\ell} K$ and
  $\Gout^\ell = \dout K \cup \db^{+,\ell} K$. Therefore, the full
  trace maps $\btin = (\btin^1,\ldots, \btin^m)$ and
  $\btout = (\btout^1,\ldots, \btout^m)$ are continuous on $\bW$.  Set
  $\bV = \ker(\btin)$ and $\bV^* = \ker( \btout).$ Then the following 
  variational equation for $\bu \in L^2(K)^m$ and $\bq = \bD \bz$ 
  with $\bz \in \bV$, 
  \begin{equation}
    \label{eq:27}
  -(\bu,\bA \bv)_{L^2(K)^m} + \ip{ \bq, \bv }_{\bW} = F(\bv),\qquad
  \forall \bv \in \bW,
  \end{equation}
  splits into $m$ decoupled equations, namely
  \begin{equation}
    \label{eq:28}
  -(\bu_\ell,\bA_\ell \bv_\ell)_{L^2(K)} + \ip{ \bD_\ell\bz_\ell, \bv_\ell }_{\bW_\ell} 
  = F(\bv_\ell),\qquad
  \forall \bv_\ell \in \bW_\ell, \quad\forall \ell=1,\ldots, m.
  \end{equation}
  Here $\bu_\ell \in L^2(K)$ and
  $\bz_\ell \in \bV_\ell \equiv \ker(\btin^\ell)$ are the $\ell$th
  components of $\bu$ and $\bz$, respectively.  By
  Theorem~\ref{thm:advect}, there is a unique $\bu_\ell \in L^2(K)$
  and $\bq_\ell =\bD_\ell \bz_\ell \in (\bV_\ell^*)^\perp$
  solving~\eqref{eq:28} for each $\ell$. This in turn proves the
  wellposedness of~\eqref{eq:27}.

  To transfer these results for  $\bA$ to $A$, we define 
  \[
  \tin w = \btin (Q^t w),\qquad
  \tout w = \btout (Q^t w).
  \]
  Then~\eqref{eq:VV*} is the same as 
  $V = \ker(\tin)$ and $V^*=\ker(\tout)$. Note that $z \in V$ if
  and only if $\bz = Q^t z \in \bV.$ Also note that a $\bu \in L^2(K)^m$ and
  $\bz \in \bV$ solves~\eqref{eq:27} if and only if $u = Q\bu$ and
  $z = Q\bz$ satisfies
  \[
  -(u, Av)_{L^2(K)^m} + \ip{ Dz, v }_W = F(Q^t v),\qquad
  \forall v \in W.
  \]
  Here we have used $(\bu, \bA \bv)_{L^2(K)^m} = (u, A v)_{L^2(K)^m}$
  and consequent identities for the corresponding boundary
  operators. Thus the stated wellposedness of~\eqref{eq:uwp-V} follows
  from the established wellposedness of~\eqref{eq:27}. The proof of
  wellposedness of~\eqref{eq:uwp-V-adj} is similar.
\end{proof}

\begin{remark}
  Consider a tent $K$ with empty $\db K$. Then, under the assumptions
  of Theorem~\ref{thm:hyperbol}, a function in $V$ has all its $m$
  components equal to zero on the inflow boundary $\din K$. Moreover,
  if $v \in V \cap C(\bar K)$, then applying the additional continuity
  property~\eqref{eq:29} to the operators $\btout^\ell$ in the above
  proof, we find that the outflow trace of each component of $v$ must
  approach zero as we approach $\Gio$ where the inflow and outflow
  boundary parts meet.
\end{remark}

\section{The wave equation}   \label{sec:wave}

We now apply the previous ideas to the important example of the wave
equation and work out the resulting weak formulation in detail. Our model
problem is to find a real-valued function $\phi$ on the space-time
domain $\om = (0,S) \times (0, T)$, satisfying
\begin{subequations} \label{eq:waveIBVP}
\begin{align}
  c^{-2} \d_{tt} \phi - \d_{xx} \phi
  & = g,  && 0<x < S, \; 0 < t < T, \\
  \d_t \phi = \phi & = 0, && t=0, \; 0 < x < S, \\
  \label{eq:imped1}
  \d_t \phi - c \,\d_x \phi & = 0, && x=0, \; 0 < t < T, \\
  \label{eq:imped2}
  \d_t \phi + c \,\d_x \phi & = 0, && x=S, \; 0 < t < T,
\end{align}
\end{subequations}
where $c>0$ is the wave speed.  Here, we have imposed the outgoing
impedance boundary conditions (but other boundary conditions can also
be considered -- see Section~\ref{sec:numerical}).

The above second order system for $\phi$ arises from a system of first order 
physical principles, which also matches 
the form of the problems we have been studying, 
namely~\eqref{eq:Au=f}. Set 
\[
u =
\begin{bmatrix}
  c\,\d_x \phi \\
  \d_t \phi
\end{bmatrix}
\]
and observe that $\d_t u_1 = c\, \d_{xt}\phi = c\,\d_x u_2$ and
$\d_t u_2 =\d_{tt}\phi = c\, \d_x u_1 + c^2 g$. These two equations
give the first order system $ Au = f$ where
\begin{equation}
  \label{eq:Awave}
  A u = \d_t u -
  \begin{bmatrix}
    0 & c \\
    c & 0
  \end{bmatrix}
  \d_x u,
  \qquad
  f =
  \begin{bmatrix}
    0\\ c^2 g
  \end{bmatrix}.  
\end{equation}
It fits into the framework of~\S\ref{ssec:1dgeneral} after the
diagonalization
\[
C \equiv
  -\begin{bmatrix}
    0 & c \\
    c & 0
  \end{bmatrix}
 = Q \Lambda Q^t, \quad
Q = \frac{1}{\sqrt{2}}
  \begin{bmatrix}
    1 & \phantom{-}1 \\
    1 & -1
  \end{bmatrix},
\quad
\Lambda =
\begin{bmatrix}
  \lambda_1 & \phantom{-}0 \\
  0 & \lambda_2
\end{bmatrix}.
\]
where $\lambda_1 =-c$ and $\lambda_2=c$.

Analogous to \eqref{eq:dtoutinbdry}, we define $\din \om = (0,S) \times \{0\}$,
$\dout \om = (0,S) \times \{T\}$ and $\db \om = \d \om \setminus (
\din \om \cup \dout \om).$ The vertical parts $\db \om$ are further
split into
\[
\db^{+,1}\om= \db^{-,2}\om = \{0\} \times [0,T],
\qquad
\db^{-,1}\om= \db^{+,2}\om = \{S\} \times [0,T].
\]
Set $\Gin^\ell$ and $\Gout^\ell$ to  the closures of
$ \din\om \cup \db^{-,\ell}\om$ and $\dout\om \cup \db^{+,\ell}\om$, respectively,
$\Gio^\ell = \Gin^\ell \cap \Gout^\ell$, and 
$\delta_\ell(x,t) = \dist( (x,t), \Gio^\ell)$ for $\ell=1,2.$ 
By a minor modification of the arguments in
Section~\ref{sec:examples}, one can prove that the global trace maps
\begin{align*}
\tin
\begin{bmatrix}
  z_1 \\ z_2
\end{bmatrix}
& =
\begin{bmatrix}
  (z_1 + z_2) \big|_{ \Gin^1} \\
  (z_1 - z_2) \big|_{ \Gin^2}
\end{bmatrix},
\qquad
\tin: W \to L^2_{\delta_1}(\Gin^1) \times L^2_{\delta_2}(\Gin^2)
\\
\tout
\begin{bmatrix}
  z_1 \\ z_2
\end{bmatrix}
& =
\begin{bmatrix}
  (z_1 + z_2) \big|_{ \Gout^1} \\
  (z_1 - z_2) \big|_{ \Gout^2}
\end{bmatrix},
\qquad
\tout: W \to
L^2_{\delta_1}(\Gout^1) \times L^2_{\delta_2}(\Gout^2)
\end{align*}
are continuous. Set
\[
V(\om) = \ker(\tin), \qquad V^*(\om) = \ker (\tout).
\]
These spaces can be used to give a global weak formulation on $\om$,
but our focus in on local solvers.

In space-time tent pitching methods, we are required to numerically
solve the wave equation on space-time tents, ordered so that inflow
data on a tent can be provided by the outflow solution on previously
handled tents or through given data. Hence we now focus on the
formulation and discretization on one tent~$K$.

\subsection{Weak formulation on a tent}

Consider the analogue of~\eqref{eq:waveIBVP} on one tent $K$, with zero
initial data on the inflow boundaries and with boundary conditions
inherited from the global boundary
conditions~\eqref{eq:imped1}--\eqref{eq:imped2}.

Define, as before, the boundary parts of a tent $K$, by
\begin{gather*}
\din K    = \{ (x,t) \in \d K: n_t<0\},
\quad
\dout K   = \{ (x,t) \in \d K: n_t>0\},
\quad
\db K     = \d K \setminus ( \din K \cup \dout K),
\\
\db^{+,1} K =
\db^{-,2} K
 = \{ (x,t) \in \db K: c n_x <0\},
\\
\db^{+,2} K = \db^{-,1} K
 = \{ (x,t) \in \db K: c n_x >0\}.
\end{gather*}  
Note that the boundary part $\db K$ may be empty in some tents. We
consider the tent problem of solving for $u$ satisfying
\begin{align*}
  Au & = f  && \text{ on } K, 
  & 
  u_1 - u_2 & = 0 &&\text{ on } \db^{+,1} K,
  \\
  u & = 0 && \text{ on } \din K,
  & 
  u_1 + u_2 & = 0 && \text{ on } \db^{+,2} K.
\end{align*}
To obtain a well-posed weak formulation on one tent, we proceed to use
Theorem~\ref{thm:hyperbol}.

To this end, we must assume that the tent satisfies
Assumption~\eqref{asm:inoutassumesystem}, which now reads
\begin{subequations}
  \label{asm:inout-wave}
\begin{align}
  \din K & \subseteq \{ x \in \d K:  n_t   \pm n_x c <0\},
  \\
  \dout K & \subseteq \{ x \in \d K: n_t   \pm  n_x c>0\}.
\end{align}
\end{subequations}
Since $\At = -A$ in this example, the weak
formulation~\eqref{eq:uwp-V} reads
\begin{equation}
  \label{eq:12}
  \begin{aligned}
    u \in L, \; q \in \Vsp:
    \qquad &-(u,A v)_L + \ip{ q,v }_W = F(v),\qquad \forall v \in W,
  \end{aligned}
\end{equation}
where the spaces are set following~\eqref{eq:VV*}, namely
\begin{align*}
    V & = \left\{
    \begin{bmatrix}
      z_1 \\ z_2
    \end{bmatrix}
    \in W:
      \begin{bmatrix}
        z_1+ z_2 \\ z_1 - z_2
      \end{bmatrix}_\ell\bigg|_{ \din K \cup \db^{-,\ell} K}
      = 0, \quad
    \text{ for } \ell=1,2\right\},
    \\
    V^* & = \left\{
    \begin{bmatrix}
      z_1 \\ z_2
    \end{bmatrix}\in W:
      \begin{bmatrix}
        z_1+ z_2 \\ z_1 - z_2
      \end{bmatrix}_\ell\bigg|_{ \dout K \cup \db^{+,\ell} K} = 0, \quad
    \text{ for } \ell=1,2
    \right\}.
\end{align*}
Theorem~\ref{thm:hyperbol} shows that~\eqref{eq:12} is a well-posed
weak formulation on $K$ provided the tent $K$
satisfies~\eqref{asm:inout-wave}.  Note that the above spaces change
from tent to tent and may arguably be better denoted by
$V(K), V^*(K)$, etc., but to avoid notational bulk we will suppress
the $K$-dependence.

\subsection{CFL condition}

Let us take a closer look at~\eqref{asm:inout-wave}.  First note that
each tent, in this application, consists of either two triangles (on
either side of the tent pole), or just one triangle. The tents are thus
divided into three types, as shown in Figure~\ref{fig:macroelt}.

\begin{figure}
  \centering
  \parbox{0.29\textwidth}{
  \begin{tikzpicture}

    \coordinate (top) at  (0,4);
    \coordinate (bot) at  (0,0);
    \coordinate (lft) at  (-1.6,1);
    \coordinate (lftp) at (0,   1);
    \coordinate (rgt) at  (1.4,1.7);
    \coordinate (rgtp) at (0,  1.7);
    \fill (top) circle (2pt);
    \fill (bot) circle (2pt);
    \fill (lft) circle (2pt);
    \fill (rgt) circle (2pt);

    \draw (top) -- (lft) -- (bot) -- (rgt) -- cycle;
    \draw[<->,dotted] (lft) -- (lftp) node [midway,above] {$h_l$};
    \draw[<->,dotted] (rgt) -- (rgtp) node [midway,above] {$h_r$};
    \draw[<->,dotted] (top) -- (bot)  node [midway,left] {$k$};
    \draw[<->,dotted] ($(lftp)-(2,0)$) -- ($(top)-(2,0)$)
                      node [midway,right] {$p_l k$};
    \draw[<->,dotted] ($(rgtp)+(2,0)$) -- ($(top)+(2,0)$)
                      node [midway,left] {$p_r k$};
  \end{tikzpicture}

  Type~I: $h_r>0, h_l >0$
  }\qquad
  \parbox{0.25\textwidth}{
  \begin{tikzpicture}

    \coordinate (top) at  (0,4);
    \coordinate (bot) at  (0,0);
    \coordinate (rgt) at  (1.4,2.2);
    \coordinate (rgtp) at (0,  2.2);
    \fill (top) circle (2pt);
    \fill (bot) circle (2pt);
    \fill (rgt) circle (2pt);

    \draw (top) -- (bot) -- (rgt) -- cycle;
    \draw[<->,dotted] (rgt) -- (rgtp) node [midway,above] {$h_r$};
    \draw[<->,dotted] (top) -- (bot)  node [midway,left] {$k$};
    \draw[<->,dotted] ($(rgtp)+(2,0)$) -- ($(top)+(2,0)$)
                      node [midway,left] {$p_r k$};
  \end{tikzpicture}

  Type~L: $h_r>0, h_l =0$
  }\qquad
  \parbox{0.26\textwidth}{
  \begin{tikzpicture}

    \coordinate (top) at  (0,4);
    \coordinate (bot) at  (0,0);
    \coordinate (lft) at  (-1.6,1);
    \coordinate (lftp) at (0,   1);
    \fill (top) circle (2pt);
    \fill (bot) circle (2pt);
    \fill (lft) circle (2pt);

    \draw (top) -- (lft) -- (bot) -- cycle;
    \draw[<->,dotted] (lft) -- (lftp) node [midway,above] {$h_l$};
    \draw[<->,dotted] (top) -- (bot)  node [midway,left] {$k$};
    \draw[<->,dotted] ($(lftp)-(2,0)$) -- ($(top)-(2,0)$)
                      node [midway,right] {$p_l k$};
  \end{tikzpicture}

  Type~R: $h_r=0, h_l >0$
  }\quad
  \caption{Three types of tents}
  \label{fig:macroelt}
\end{figure}
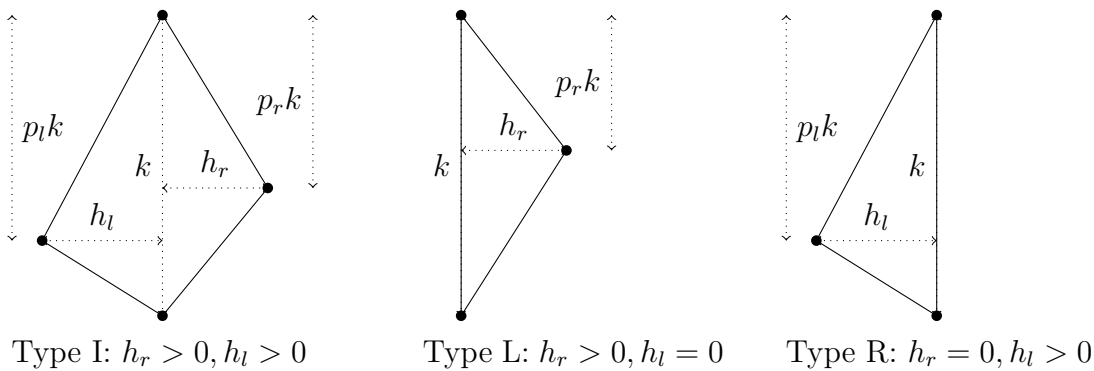

The length of the tent pole is $k$, the numbers $p_l$ and $p_r$ are
such that $p_lk$ and $p_rk$ give the heights of the outflow boundaries
on the left and right side of the tent pole, respectively, and the
spatial mesh size are $h_r, h_l\ge 0$. Writing down the normal vector
on the tent boundaries, we immediately find that
condition~\eqref{asm:inout-wave} on a tent is equivalent to
\begin{equation}
  \label{eq:cfl-2}
\left| c \frac{ k p_r}{ h_r} \right| <1
\quad\text{ and }\quad
\left| c \frac{ k p_l}{ h_l} \right| <1.
\end{equation}
Clearly, by controlling the size of the tent pole we can satisfy
these inequalities.

The well-known Courant-Friedrichs-Levy (CFL)
condition~\cite{CouraFriedLewy28} identifies stability conditions as
constraints on the time step size in terms of space mesh size in
numerical discretizations. In our case, this condition manifests
itself as geometrical constraints~\eqref{eq:cfl-2} on the tent.  For
this reason, we will refer to~\eqref{asm:inout-wave}
-- or~\eqref{eq:cfl-2} -- as the CFL condition of our method.

\section{The numerical scheme}   \label{sec:scheme}

In this section, continuing to work with the wave operator $A$ defined
by~\eqref{eq:Awave}, we give an explicit numerical scheme for
approximating $u(x,t)$ satisfying
\begin{subequations}  \label{eq:theproblem}
  \begin{align}
    A u & = f,     &&  0<x<S, \; 0<t<T,
    \\
    u_i(x,0) & = u_i^0(x), && 0< x< S, \; i=1,2,
   \\
    u_1(0,t) - u_2(0,t) & = 0, && 0< t< T,
    \\                       
    u_1(1,t) + u_2(1,t) & = 0,  && 0< t< T.    
  \end{align}
\end{subequations}
The scheme will allow varying spatial and temporal mesh sizes.  Here $f$ and
$u^0$ are assumed to given smooth functions. We begin by describing
the calculations within each tent, followed by the tent pitching
technique to advance in time.

\subsection{Conforming discretization on a tent}   \label{ssec:confomingtent}

As seen above, a tent is comprised of one or two triangles. Let the
space of continuous functions on a tent $K$ whose restrictions to
these triangles are linear be denoted by $P_1^h(K)$.  We construct a
conforming discretization of~\eqref{eq:12} within $K$ using the
discrete space
\begin{equation}
  \label{eq:13}
V_1 = V \cap ( P_1^h(K) )^2.
\end{equation}
By definition, $V_1 \subseteq V$, and consequently, functions in $V_1$ must
satisfy the essential boundary conditions of $V$.  Depending on the
tent geometry, different boundary conditions must be imposed on
different tents.

To examine what this entails for the nodal coefficients on mesh
vertices, let $\zeta \in P_1^h(K)$ be the continuous scalar function
(unique Lagrange basis function) that equals one at the ``apex'' of
the tent $K$, equals zero at all its other vertices. The {\em apex} of
a tent, irrespective of whether it consists of one or two triangles,
is the vertex in $\dout K$ that is away from $\din K$.  Now, suppose
$\mu
\equiv [\begin{smallmatrix} \mu_1 \\
  \mu_2
\end{smallmatrix}]$ in $\RRR^2$ is such that 
\begin{subequations}
  \label{eq:mu}
\begin{align}
  \mu_1 - \mu_2 & =0 && \text{ if $\db^{+,1} K$
                       is nonempty},
\\
  \mu_1 + \mu_2 & = 0 && \text{ if $\db^{-,1} K$
                       is nonempty}.
\end{align}
\end{subequations}
(Note that if $\db K$ is empty, then $\mu$ is an arbitrary vector in
$\RRR^2$.)  Then, it is easy to see that
\begin{equation}
  \label{eq:14}
V_1 = \left\{ \mu \zeta:  \; \mu \text{ satisfies~\eqref{eq:mu}}
 \right\}
\end{equation}
provides an alternate characterization of~\eqref{eq:13}.

A computable conforming discretization of~\eqref{eq:12} additionally
requires finite-dimensional subspaces of $L$ and $\Vsp$. For the
latter, observe that~\eqref{eq:10} implies that
\[
D(V_1) \subset D(V) = \Vsp.
\]
Hence we choose an approximation $q_1$
of the solution component $q$ in~\eqref{eq:12} to have the form 
\[
q_1 = D z_1, \qquad z_1 \in V_1.
\]
Then $q_1$ is clearly in $\Vsp.$ Next, set $L_1 \subset L$ to be the space
of vector functions whose components are constant functions on
$K$. Finally, set
\begin{equation}
  \label{eq:W}
  W_1 = \left\{ w: \;
  w = \kappa
  + \mu
  \zeta, \quad \kappa, \mu \in \RRR^2, \quad
  \kappa \zeta \in V_1
  \right\}.
\end{equation}

Our discretization of~\eqref{eq:12} now takes the following form: Find
$u_1$ in $L_1$ and $q_1 \in D(V_1)$ satisfying
$ -(u_1,A w) + \ip{ q_1, w} = F(w),$ for all $w \in W_1.$ Clearly,
$\dim(W_1)$ is four or three, depending on whether $\db K$ is empty or
not. This equation gives rise to an invertible discrete system, as a
consequence of the unisolvency of the following slightly modified
problem:
\begin{equation}
  \label{eq:discrete}
  \begin{aligned}
    \text{\em Find $u_1\in L_1$} &\text{ \em and $z_1 \in V_1$
      such that}
    \\
    \qquad &-(u_1,A w) + \ip{ D z_1, w} =
    F(w),
    \qquad
    \forall w \in  W_1.
  \end{aligned}
\end{equation}

\begin{proposition}
  There is a unique solution for Problem~\eqref{eq:discrete}.
\end{proposition}
\begin{proof}
  Note the $\dim(L_1) + \dim( V_1) = \dim(W_1)$,
  so~\eqref{eq:discrete} gives a square (Petrov-Galerkin)
  system.  Hence it suffices to set $F=0$ and
  prove that $u_1=z_1=0$. With $F=0$, writing $z_1 = \alpha \zeta$ for
  some $\alpha \in \RRR^2$, we have $\ip{ D (\alpha \zeta), w} =0$ for
  all constant $w \in W$. Since $\alpha \zeta \in V_1$, by the
  definition of $W_1$, we may set $w = \alpha$ in~\eqref{eq:discrete},
  to get
  \[
  \ip{ D (\alpha \zeta), \alpha }
  = \int_{\dout K \cup \db K} \Dd \alpha \cdot \alpha \zeta
  = 0
  \]
  where
  \begin{equation}
    \label{eq:Dd}
    \Dd = n_t I + C n_x =
    \begin{bmatrix}
      n_t & -cn_x \\
      -cn_x& n_t
    \end{bmatrix}.
  \end{equation}
  If $\db K$ is empty, then since $ \Dd
  \alpha \cdot \alpha = (\alpha_1 + \alpha_2)^2 (n_t - cn_x)/2 +
  (\alpha_1 - \alpha_2)^2 (n_t + cn_x)/2,$ the CFL
  condition~\eqref{asm:inout-wave} gives $\alpha=0$.  If $\db K$ is
  nonempty, then whenever $\alpha \zeta \in V$ we have either
  $\alpha_1 - \alpha_2 =0$ or $\alpha_1 + \alpha_2 =0$, so we can
  continue to conclude that $\alpha=0$. Of course $\alpha=0$ implies $z_1=0$.

  To prove that $u_1=0$, we use~\eqref{eq:D} after substituting $z_1=0$, to get
  \[
  \ip{ D w, u_1} =0, \qquad \forall w \in W_1.
  \]
  Since $u_1$ is a constant function, $u_1 \zeta \in W_1$, so we may
  choose $w = u_1\zeta$ and conclude that $u_1=0$ by an argument
  analogous to what we used above.
\end{proof}

\begin{remark}
  One can view $z_1|_{\dout K}$ as an interface trace variable and
  $q_1 = Dz_1$ as an interface flux variable. By the trace theory
  we developed previously, outflow trace $z_1|_{\dout K}$ must vanish
  at the points where outflow and inflow edges meet in order for $z_1$
  to be in $V$. This motivates our choice~\eqref{eq:14} of $V_1$ to
  obtain a conforming method. Other non-conforming avenues to design
  approximations within a tent can be found in~\cite{FalkRicht99}
  and~\cite{MonkRicht05}.
\end{remark}

\subsection{Advancing in time by tent pitching}   \label{ssec:tent-pitch-algo}

We now show how the above ideas yield an explicit time marching
algorithm for solving~\eqref{eq:theproblem}.
First, we mesh the space-time domain $\om = (0,S) \times (0, T)$ by a
collection $\oh$ of tents $K$ with these properties: The first
property is that either $\db K$ is empty or
\begin{subequations}
  \label{eq:mesh}
  \begin{equation}
    \label{eq:21}
     \db K \subseteq \db\om,
  \end{equation}
  for all $K \in \oh$.
  Second, there exists an enumeration of all tents, $K_1, K_2,
  \ldots, K_J$, with the property that for each $j \in \{ 1,\ldots
  J\}$,
\begin{equation}
  \label{eq:downwind-2}
  \din K_j \subseteq \union_{k=1}^{j-1} \dout K_k \cup \din \om.
\end{equation}
Finally, for all $j \in \{ 1,\ldots J\}$,
\begin{equation}
  \label{eq:15}
  K_j \text{ satisfies the CFL condition~\eqref{eq:cfl-2}.}
\end{equation}
\end{subequations}
It is well-known how to construct an algorithm (not only in one space
dimension, but also in higher
dimensions~\cite{ErickGuoySulli05,UngorSheff02}) that produces meshes
satisfying~\eqref{eq:mesh}, so we shall not dwell further on the
meshing process.

The discrete space-time approximation on the mesh $\oh$ is developed using 
\begin{subequations}\label{eq:spaces}
\begin{align}
V_h & = \big\{ z
  \in  H^1(\om)^2 \cap V(\om) :
  \; z|_K \in P_1^h(K)^2,\; \forall K \in \oh, 
    \text{ and }      z(x,0) = I_h u^0(x) \big\}
\\
L_h & = \left\{
  \alpha : \; \alpha|_K \in \RRR^2 \text{ is constant on each } K \in \oh
\right\}
\\
W_h & = \left\{
  w : \; w|_K \in W_1 \text{ on each } K \in \oh
\right\},
\end{align}  
\end{subequations}
where $I_h$ denote the linear nodal interpolant on the spatial mesh.
The method finds approximations $u_h \in L_h$ and $z_h \in V_h$
satisfying
\begin{equation}
  \label{eq:composite}
\sum_{K \in \oh} \left(
-\int_K u_h \cdot A w + \int_{\d K} \Dd z_h \cdot w\right)
= \sum_{K\in \oh}  \int_K f \cdot w,
\qquad \forall w \in W_h,
\end{equation}
where $\Dd$ is as in~\eqref{eq:Dd}.

Because of~\eqref{eq:mesh}, we are able to use a time-marching
algorithm to solve~\eqref{eq:composite}: Proceed in the ordering
of~\eqref{eq:downwind-2}, and for each tent~$K$, solve for $u_h|_K$ and
$z_h|_K$. Specifically, if $\alpha$ is the nodal (vector) value of $z_h$ at the
apex of $K$, then defining $\zout = \alpha \zeta$, the problem on one
tent is to find $u_h|_K \in L_1$ and $\zout \in V_1$
satisfying~\eqref{eq:composite}, namely
\begin{equation}
  \label{eq:20}
-\int_K u_h \cdot A w + \int_{\d K} \Dd \zout  \cdot w
= \int_K f \cdot w - \int_{\d K} \Dd \zin  \cdot w,
\qquad\forall w \in W_1.
\end{equation}
where $\zin = z_h - \zout$.  Note that $\zin$ on right hand side will
be a known quantity if~\eqref{eq:downwind-2} holds and if we have
already solved on every $K'$ appearing before~$K$ in the ordering of
tents in~\eqref{eq:mesh}. Indeed, $\zin$ is completely determined by
its nodal values at (the three or two) vertices on $\din K$, which
either lie at $t=0$ or were apex vertices of previous
tents. Problem~\eqref{eq:20} is exactly of the same type we discussed
in \S~\ref{ssec:confomingtent}.

\subsection{Propagation formula}

Since the system~\eqref{eq:20} is small, we can explicitly calculate
its solution. To see how information is propagated from inflow to
outflow on a mesh of tents, we consider the case where the volume
source $f$ is zero. Write $z_h = \left[
  \begin{smallmatrix}
    z_{h,1} \\ z_{h,2}
  \end{smallmatrix}
\right]$ in~\eqref{eq:composite} and let the nodal values of the
scalar Lagrange finite element functions $z_{h,1}$ and $z_{h,2}$ be
$\left[
  \begin{smallmatrix}
    U^t\\ V^t
  \end{smallmatrix}
\right],
\left[
  \begin{smallmatrix}
    U^b\\ V^b
  \end{smallmatrix}
\right],
\left[
  \begin{smallmatrix}
    U^l\\ V^l
  \end{smallmatrix}
\right],
\left[
  \begin{smallmatrix}
    U^r\\ V^r
  \end{smallmatrix}
\right],$ at the top, bottom, left and right vertices, respectively,
of a tent of Type~I, as in Figure~\ref{fig:macroelt}. For the
other two tent types, we omit the nodal values at the
missing vertex.

Equation~\eqref{eq:20} finds $\left[
  \begin{smallmatrix}
    U^t\\ V^t
  \end{smallmatrix}
\right]$ as a function of the remaining nodal values. After tedious
simplifications (not displayed), this relationship is found to be as
follows:
\begin{align}
  \label{eq:solution}
  \begin{bmatrix}
    U^t \\ V^t
  \end{bmatrix}
  & =
  \begin{bmatrix}
    U^b \\ V^b
  \end{bmatrix}
  +
  w_1
  \begin{bmatrix}
    0 & c \\
    c & 0
  \end{bmatrix}
  \begin{bmatrix}
    U^r - U^l \\ V^r - V^l
  \end{bmatrix}
  + w_2 c
  \begin{bmatrix}
    U^r - U^l \\ V^r - V^l
  \end{bmatrix},
&& \text{ for Type~I},
\\ \nonumber
  \begin{bmatrix}
    U^t \\ V^t
  \end{bmatrix}
  & =
  \begin{bmatrix}
    U^b \\ V^b
  \end{bmatrix}
  +
  w_1
  \begin{bmatrix}
    0 & c \\
    c & 0
  \end{bmatrix}
  \begin{bmatrix}
    U^r - U^b \\ V^r - V^b
  \end{bmatrix}
  + w_2 c
  \begin{bmatrix}
    U^r - U^b \\ V^r - V^b
  \end{bmatrix},
&& \text{ for Type~L},
\\ \nonumber
  \begin{bmatrix}
    U^t \\ V^t
  \end{bmatrix}
  & =
  \begin{bmatrix}
    U^b \\ V^b
  \end{bmatrix}
  +
  w_1
  \begin{bmatrix}
    0 & c \\
    c & 0
  \end{bmatrix}
  \begin{bmatrix}
    U^b - U^l \\ V^b - V^l
  \end{bmatrix}
  + w_2 c
  \begin{bmatrix}
    U^b - U^l \\ V^b - V^l
  \end{bmatrix},
&& \text{ for Type~R},
\end{align}
where
\begin{align*}
  w_1 & = \frac{(h_r+h_l) k}{(h_r+h_l)^2 - c^2 k^2 (p_r-p_l)^2},
  &
  w_2 & = \frac{c\,k^2(p_r-p_l)}{(h_r+h_l)^2 - c^2 k^2 (p_r-p_l)^2},
  && \text{ for Type~I},
  \\
  w_1 & = \frac{k}{2 ( ck (1-p_r) + h_r) },
  &
  w_2 & = w_1,
  && \text{ for Type~L},
  \\
  w_1 & = \frac{k}{2 ( ck (1-p_l) + h_l) },
  &
  w_2 & = -w_1,
  && \text{ for Type~R}.
\end{align*}

\subsection{Error analysis on uniform grids}   \label{ssec:unifanal}

We now work out the stencil given by the method on a uniform grid
where all tents are shaped the same (see
Figure~\ref{fig:stencil}). The stencil translates~\eqref{eq:solution}
into an equation that gives the nodal values of the outflow apex
vertex, given the nodal values at the inflow vertices. Let $h>0$ be
the uniform spatial mesh size, $k>0$ be the time step size measured,
as before, by the height of the tent pole.  At a point $(jh/2, kn/2)$
in the lattice $(h/2) \ZZZ\times (k/2) \ZZZ$, let $(U_j^n, V_j^n)$
denote the nodal value of the approximation to $z_h$ there.  As shown
in Figure~\ref{fig:stencil}, the scheme uses only a subset of lattice
points in $h \ZZZ\times k \ZZZ$.  Each grid point involved in the
scheme has an associated $U$ value (indicated in the figure by
``\tikz\fill (0,0) circle (2pt);'') and a $V$ value (indicated by
``\tikz\draw (0,0) circle (2pt);'').

\begin{figure}
  \centering
  \begin{tikzpicture}
    \draw [step = 1cm,dotted] (-2,2) grid (9,7);

    \foreach \n in {1,2,3} {
      \foreach \j in {-1,0,1,2,3,4} {

        \fill (2*\j+1-0.05,2*\n) circle (2pt);
        \draw (2*\j+1+0.05,2*\n) circle (2pt);

        \fill (2*\j-0.05,2*\n+1) circle (2pt);
        \draw (2*\j+0.05,2*\n+1) circle (2pt);
      }
    }


    \draw[<->] (2*4,2*1 -0.3) -- (2*4+1,2*1 -0.3)
                node[midway,below] {$h/2$};
    \draw[<->] (2*4+1 +0.3,2*1+1) -- (2*4+1 +0.3,2*1)
                node[midway,right] {$k/2$};

    \filldraw[-,very thick,blue,opacity=0.4]
               (2*3-1,2*3) -- (2*3,2*3-1) --
               (2*3+1,2*3) -- (2*3,2*3+1) -- cycle;
    \draw[-,very thick,blue,opacity=0.4]
               (2*3,2*3-1) -- (2*3,2*3+1);
    \node at (2*3+1.4,2*3) {\rotatebox{-90}{tent}};

    \draw[-] (2*1-1,2*2) -- (2*1+1,2*2) ;
    \draw[-] (2*1,2*2-1) -- (2*1,2*2+1) ;

    \node at (2*1-1.7,2*2) {
      \footnotesize{
      $
      \begin{bmatrix}
        U_{j-1}^{n}\\
        V_{j-1}^{n}
      \end{bmatrix}
      $}};
    \node at (2*1+1.6,2*2) {
      \footnotesize{
      $
      \begin{bmatrix}
        U_{j+1}^{n}\\
        V_{j+1}^{n}
      \end{bmatrix}
      $}};
    \node at (2*1,2*2+1.7) {
      \footnotesize{
      $
      \begin{bmatrix}
        U_{j}^{n+1}\\
        V_{j}^{n+1}
      \end{bmatrix}
      $}};
    \node at (2*1,2*2-1.7) {
      \footnotesize{
      $
      \begin{bmatrix}
        U_{j}^{n-1}\\
        V_{j}^{n-1}
      \end{bmatrix}
      $}};

  \end{tikzpicture}
  \caption{The stencil}
  \label{fig:stencil}
\end{figure}
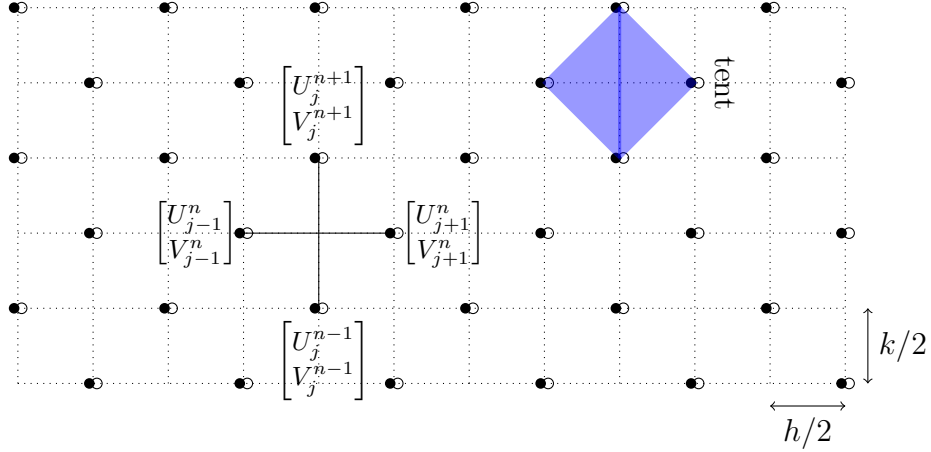

Equation~\eqref{eq:solution} now simplifies to
\begin{subequations}  \label{eq:unifTP}
\begin{align}
  U^{n+1}_j
  & = U^{n-1}_j + a c ( V_{j+1}^n - V_{j-1}^n )
  \\
  V^{n+1}_j
  & = V^{n-1}_j + a c ( U_{j+1}^n - U_{j-1}^n )
\end{align}  
\end{subequations}
where $a = k/h$. This is simply the non-staggered leapfrog scheme (studied
extensively for scalar equations) applied to the first order
system. By a simple Taylor expansion about the stencil center, we see
that the scheme is consistent and that the local truncation error is
of second order (see~\cite{RichtMorto94,Strik89} for definitions of
these and related terminology).

To examine stability, introduce a new vector variable $X_j^n$ and
rewrite the scheme~\eqref{eq:unifTP} as follows:
\begin{align*}
  X^{n+1}_j =
  \begin{bmatrix}
    [X_j^n]_3 + ac [X_{j+1}^n - X_{j-1}^n]_2 \\
    [X_j^n]_4 + ac [X_{j+1}^n - X_{j-1}^n]_1 \\
    [X_j^n]_1 \\
    [X_j^n]_2
  \end{bmatrix}, 
\qquad \text{ where }
X_j^n  =
\begin{bmatrix}
  U_j^n \\ V_j^n\\ U_j^{n-1}\\ V_j^{n-1}
\end{bmatrix}.
\end{align*}
To this one-step scheme, we now 
apply von Neumann analysis~\cite{RichtMorto94,Strik89}. The
amplification matrix $G$, connecting $X_j^{n+1}$ to $X_j^n$ can be
readily calculated:
\[
G=
\begin{bmatrix}
  0 & 2\ii s & 1 & 0 \\
  2\ii s &0 &0& 1 \\
  1 & 0 & 0 & 0\\
  0 & 1 & 0 & 0
\end{bmatrix}
= R \Lambda R^{-1}, \quad\text{ where }\quad
R =
\begin{bmatrix}
1 & 1 & 1 & 1\\
1 & 1 & −1 & −1\\
g^{-1}_1 & g^{-1}_2 & g^{-1}_3 & g^{-1}_4\\
g^{-1}_1 & g^{-1}_2 & -g^{-1}_3 & -g^{-1}_4
\end{bmatrix},
\]
$\ii$ denotes the imaginary unit, the eigenvalues of $G$ are
$g_1 =\ii s - \sqrt{1 - s^2},\; g_2 = \ii s + \sqrt{1 - s^2},\; g_3 =
- \ii s - \sqrt{1 - s^2},\; g_4 = -\ii s + \sqrt{1 - s^2}$,
$\Lambda = \diag(g_i)$, $s = ac \sin(\theta),$ and
$\theta \in [-\pi,\pi]$ gives the frequency in von Neumann
analysis. If
\begin{equation}
  \label{eq:CFL}
  | ac | < 1
\end{equation}
then all eigenvalues satisfy $|g_i|=1$. Furthermore, since
$\det R = 4(g_1^{-1} - g_2^{-1}) (g_4^{-1} -g_3^{-1})$ remains away
from zero whenever~\eqref{eq:CFL} holds, the powers
$G^n = R \Lambda^n R^{-1}$ are uniformly bounded for all $n$ and all
$\theta$. Hence~\eqref{eq:CFL} implies that the scheme is stable.

We thus conclude, by the Lax-Richtmyer theorem, that the scheme is
convergent and is of second order.  Note that the CFL condition we
previously found on general meshes, namely~\eqref{eq:cfl-2}, when
restricted to uniform meshes, gives exactly the same CFL
condition~\eqref{eq:CFL} obtained above from von Neumann analysis.

\bigskip

\section{Numerical results}   \label{sec:numerical}

\subsection{Convergence study}

\begin{figure}
  \centering
  \begin{subfigure}{0.45\textwidth}
    \includegraphics[width=\textwidth]{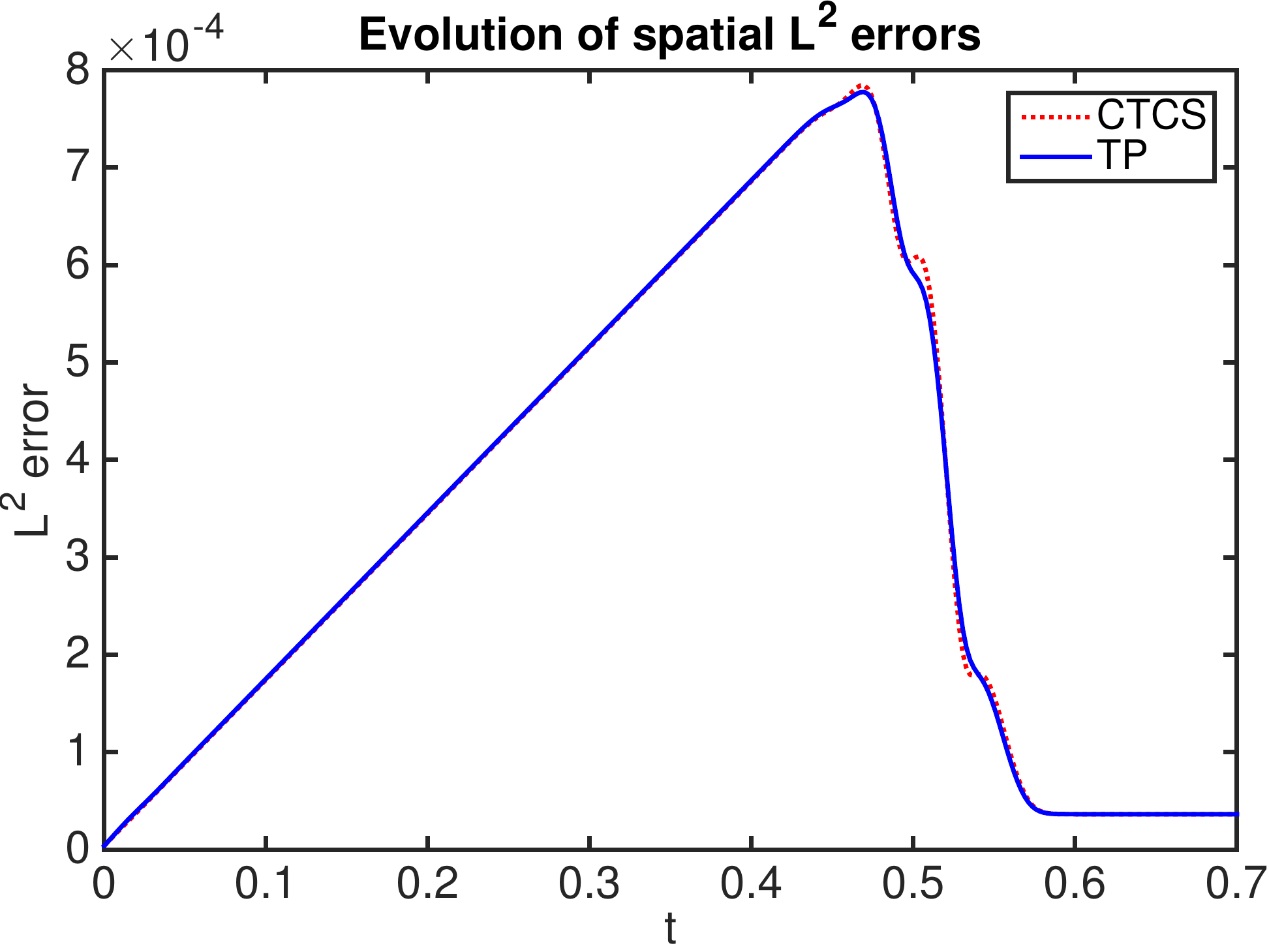}
    \caption{Errors as a function of time}
    \label{fig:compare-evolve}
  \end{subfigure}
  \begin{subfigure}{0.45\textwidth}
    \includegraphics[width=\textwidth]{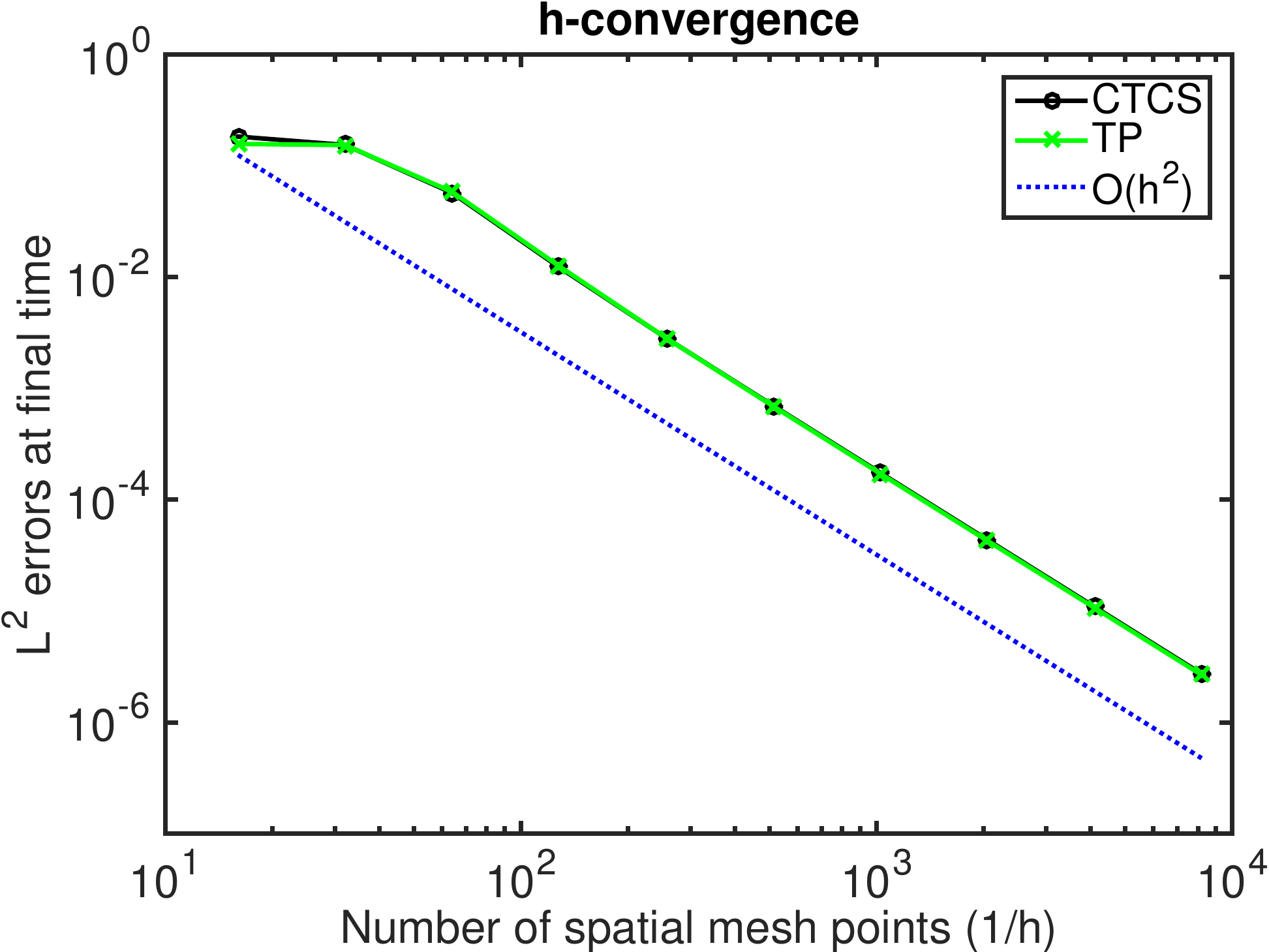}
    \caption{Second order convergence}
    \label{fig:compare-convergence}
  \end{subfigure}
  \caption{Comparison with CTCS scheme}
  \label{fig:compare}
\end{figure}

First, we report numerical results from our tent pitching (TP) scheme
and compare it with the well-known ``central-time central-space''
(CTCS) finite difference scheme
(see~\cite{CouraFriedLewy28,RichtMorto94}, sometimes also known as the
Yee scheme~\cite{Gusta08,Yee66}).  The only difference between the two
is that while the TP scheme sets the $U$ and $V$ nodes on the same
location (exactly as indicated in Figure~\ref{fig:stencil}), the CTCS
scheme sets them on staggered locations on the same grid. Both schemes
are applied to the model problem~\eqref{eq:theproblem} on uniform
grids with $S=1$. We use a grid like that in Figure~\ref{fig:stencil}
for both methods.

To impose the outgoing impedance boundary conditions within the CTCS
scheme, we use the standard finite difference technique of introducing
ghost points to the left and right of the finite grid and eliminating
the unknown values at those points using the boundary condition. In
contrast, in the TP scheme, the impedance boundary
conditions are essentially imposed within the finite element spaces,
as we have already seen previously. We set $c=1$ and impose the
initial condition so that the exact solution is
\[
u_1(x,t) = u_2(x,t) = e^{-1000 ( (x+t) - 1/2)^2 },
\]
i.e., the solution is a smooth pulse moving to the left at unit speed,
eventually clearing out of the simulation domain. At every other time
step (in the uniform space-time grid) we compute the $L^2(0,1)$-norm
of the difference between the computed and exact solution. The
evolution of these errors in time on a grid of spatial mesh size
$h=0.0025$ and $k = 0.9 h$ is shown in
Figure~\ref{fig:compare-evolve}.

We observe from Figure~\ref{fig:compare-evolve} that the errors of
both methods are comparable and remain low throughout the simulated
time.  Note also that after the pulse clears the simulation domain
reflectionlessly (and the solution within $[0,1]$ vanishes), the errors
for both methods decrease markedly. In Figure~\ref{fig:compare}, we
display a log-log plot of the $L^2(0,1)$-norm of the errors at $t=0.5$
for $h=1/2^3, \ldots, 1/2^{13},$ and $k = 0.9 h$. The rate of decrease
of this error is clearly seen to be of the order $O(h^2)$. This is in
accordance with our von Neumann analysis of \S~\ref{ssec:unifanal}
(although we did not take into account boundary conditions in that
analysis).

Thus we conclude from Figure~\ref{fig:compare} that there is
negligible difference between the performance of the two methods on
uniform grids.

\subsection{Material interfaces and other boundary conditions}

Next, we consider a generalization of~\eqref{eq:theproblem} given by 
\begin{subequations}  \label{eq:wavegen}
  
\begin{align}
  \label{eq:wavegen-pde}
  \d_t 
  \begin{bmatrix}
    \kappa_1 & 0 \\
    0 & \kappa_2
  \end{bmatrix}
  \begin{bmatrix}
    u_1 \\ u_2 
  \end{bmatrix}
  - 
  \begin{bmatrix}
    0 & c \\
    c & 0 
  \end{bmatrix}
  \d_x 
  \begin{bmatrix}
    u_1 \\ u_2 
  \end{bmatrix}
       & = f ,
      && 0< x< 1 , 0< t < T,
 \\
  u_1(x,0) & = u_1^0(x), && 0< x< 1,
\\
  u_2(x,0) & = u_2^0(x), && 0< x< 1,
  \\
  \label{eq:wavegen-bc0}
  z_0 u_1 - u_2 & = 0,  && x=0, \; 0< t < T,
  \\                          
  \label{eq:wavegen-bc1}
  z_1 u_1 + u_2 & = 0,  && x=1, \; 0< t < T.
\end{align}
\end{subequations}
where $\kappa_1(x)$ and $\kappa_2(x)$ are time-independent material
parameters and $c$, $z_0$ and $z_1$ are constants.  Such systems arise from
electromagnetics or acoustics~\cite{Gusta08,LeVeq02} on layered media
and the differential equation is often written in the following
equivalent, but non-symmetric  form
\[  
\d_t 
  \begin{bmatrix}
    u_1 \\ u_2 
  \end{bmatrix}
  - 
  \begin{bmatrix}
    0 & \beta_1 \\
    \beta_2 & 0 
  \end{bmatrix}
  \d_x 
  \begin{bmatrix}
    u_1 \\ u_2 
  \end{bmatrix}
       = 
       \tilde f 
\]
where $\beta_i(x) = c/\kappa_i(x)$ and
$\tilde f = \diag( \kappa_1^{-1},\kappa_2^{-1}) f$ obtained by scaling
the equations of~\eqref{eq:wavegen-pde} by $\kappa_1^{-1}$ and
$\kappa_2^{-1}$.  When $\kappa_1(x) \equiv \kappa_2(x)\equiv 1$ and
$z_0 = z_1 =1$, we obtain the model formulation we discussed
previously in detail.  Dirichlet boundary conditions can be imposed by
putting $z_0 = z_1 =0$, while exact outgoing impedance conditions can
be imposed using $z_0= \sqrt{\kappa_1/\kappa_2}$ and
$z_1=\sqrt{\kappa_1/\kappa_2}$. Intermediate values of $z_i$ give
damped impedance boundary conditions.

Whenever $\kappa_i$ is a constant on each spatial mesh interval, a
tent pitching scheme is suggested by a simple generalization of the
previous algorithm for homogeneous media. We define the discrete spaces 
exactly as in~\eqref{eq:spaces}, but noting  that $V(\om)$ now
has different essential boundary conditions -- stemming from
~\eqref{eq:wavegen-bc0}--\eqref{eq:wavegen-bc1} -- which are  inherited
by the spaces on tents with its tent pole on the boundary.  The
generalization of the scheme is derived by merely setting the $A$ and
$\Dd$ in~\eqref{eq:20} by
\[
A = 
  \begin{bmatrix}
    \kappa_1 & 0 \\
    0 & \kappa_2
  \end{bmatrix}
  \d_t 
  - 
  \begin{bmatrix}
    0 & c \\
    c & 0 
  \end{bmatrix}
  \d_x, 
\qquad 
\Dd = 
\begin{bmatrix}
  n_t \kappa_1 & -c n_x \\
  -c n_x & n_t \kappa_2
\end{bmatrix}.
\]
Note that this $A$, appearing on the left hand side
of~\eqref{eq:wavegen-pde}, satisfies~\eqref{eq:T}.  By solving this
general version of~\eqref{eq:20} one can obtain propagation formulas
similar to~\eqref{eq:solution}, but we omit these details and report
only the numerical results.

\begin{figure}  
  \centering
  \begin{subfigure}{0.4\textwidth}
    \includegraphics[trim=0in 0.5in 0in 0.2in, clip, width=\textwidth]{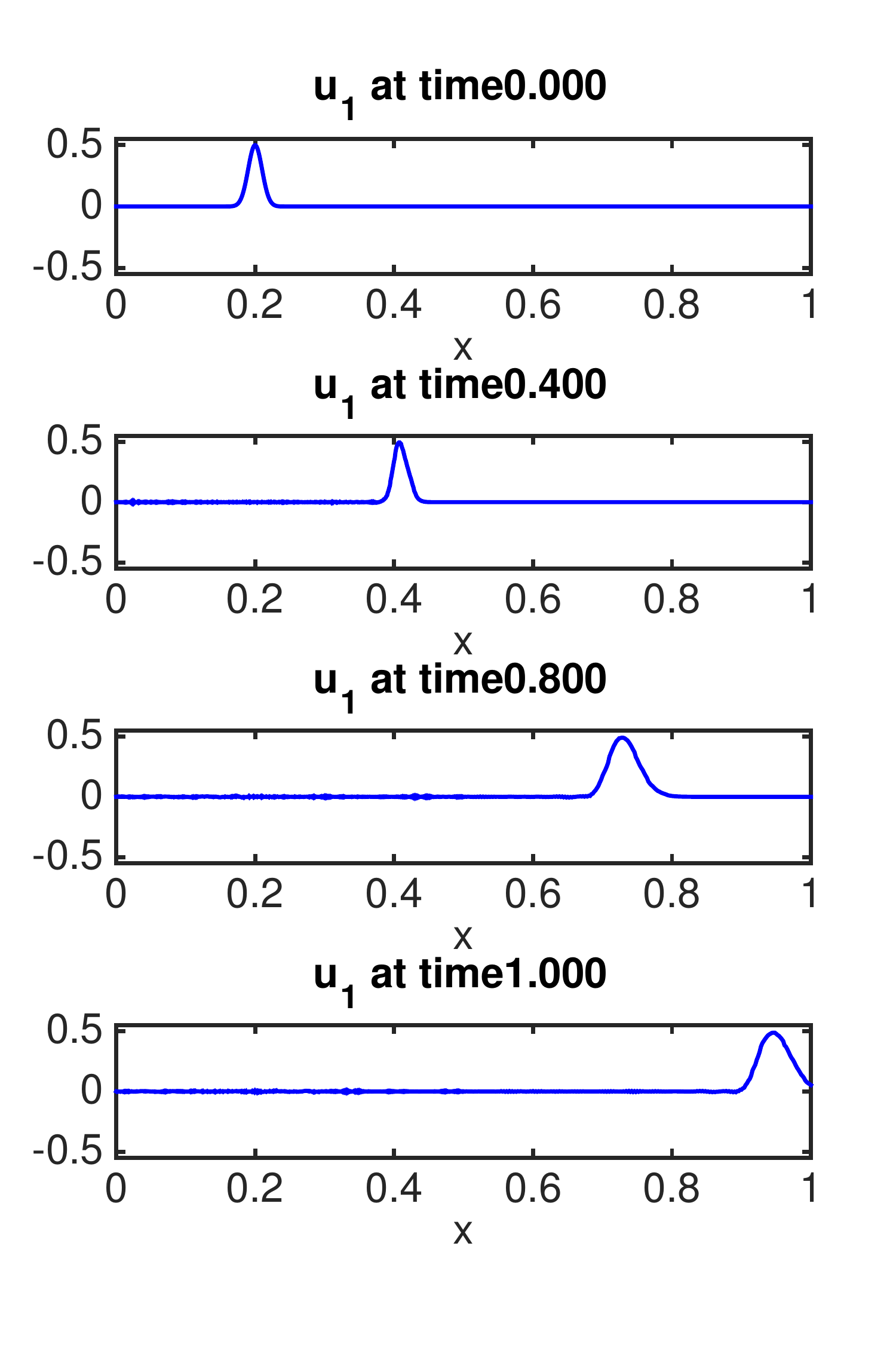}
    \caption{Time snapshots of $u_1$ (cf.\ whole 
      space-time plot of $u_1$
      in Figure~\ref{fig:patch1})}
    \label{fig:snaps1}
  \end{subfigure}\quad
  \parbox{0.49\textwidth}{
  \begin{subfigure}{0.5\textwidth}
    \includegraphics[width=\textwidth]{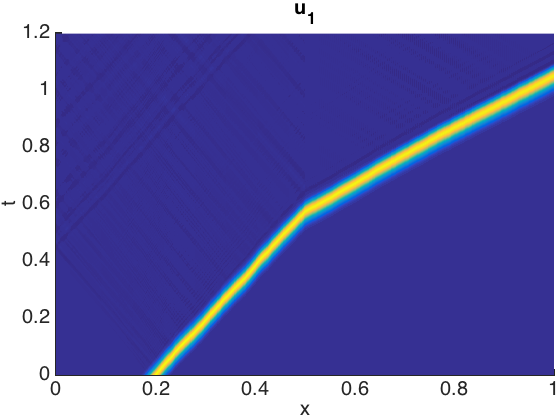}
    \caption{Entire space-time plot (over all time slabs) of the $u_1$-component 
      of the solution}
    \label{fig:patch1}
  \end{subfigure}
  \begin{subfigure}{0.5\textwidth}
    \includegraphics[trim=1in 2.5in 1in 1.8in, clip, width=\textwidth]
                   {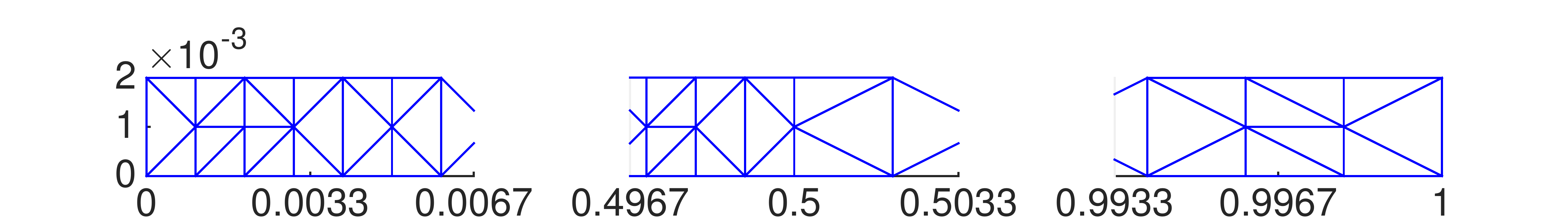}
    \caption{Parts of the tent pitched mesh of the initial 
      time slab, with varying spatial mesh sizes 
      in the regions  $x<0.5$ and $x>0.5$.}
    \label{fig:mesh1}
  \end{subfigure}
  }
  \caption{Wave propagation through an impedance-matched interface}
  \label{fig:impedancematched}
\end{figure}

First we consider the case 
\[
\kappa_1 = 
\left\{
  \begin{matrix}
    2, &&  0< x < 1/2,
    \\
    1, &&  1/2 < x < 1,
  \end{matrix}
\right.
\qquad
\kappa_2 = 
\left\{
  \begin{matrix}
    2, &&  0< x < 1/2,
    \\
    1, &&  1/2 < x < 1,
  \end{matrix}
\right.
\]
and $c=1$. The wave speed (equalling $c / \sqrt{ \kappa_1 \kappa_2}$), jumps from
$0.5$ in the left half to $1$ in the right half.  However, the
impedance (equalling $\kappa_1/\kappa_2$ -- see~\cite{LeVeq02}) is one
in both regions. Thus $x=0.5$ is an impedance-matched interface about
which we do not expect to see any reflection.

We use the tent pitching method to simulate a wave propagating to the
right starting near $x=0.2$. To this end, define a smooth pulse
$g(x) = e^{-5000(x-0.2)^2 }$ and set the data in~\eqref{eq:wavegen} by
\begin{equation}
  \label{eq:16}
f=0, \quad 
u_1^0 (x) = (c/\kappa_1) g(x), 
\quad 
u_2^0 (x) = -(c/\sqrt{\kappa_1\kappa_2}) g(x),  
\end{equation}
and $z_0= \sqrt{\kappa_1/\kappa_2}$ and
$z_1=\sqrt{\kappa_1/\kappa_2}$. We use a spatial mesh of mesh size
$h=10^{-3}$ in the left half and $h = 2\times 10^{-3}$ in the right
half. A simple tent meshing algorithm then produces a mesh of
space-time tents based on this non-uniform spatial mesh that satisfies
the CFL condition~\eqref{eq:cfl-2}. The meshing algorithm proceeds as
illustrated as in Figure~\ref{fig:tents} by simply picking a point
with the lowest time coordinate to pitch a tent. When multiple
locations have the minimal time coordinate, the algorithm picks a
tent pitching location among them randomly, thus giving an unstructured
mesh.  To minimize the overhead in constructing the mesh of tents,
instead of meshing the entire space-time domain at once, we first mesh
a thin time slab $\{ (x,t): 0<t<0.002, \; 0< x< 1\}$ and then
repeatedly stack this mesh in time to cover the entire region of time
simulation. The mesh of the initial slab is shown in
Figure~\ref{fig:mesh1}.

One of the two components of the computed solution is shown in the
remaining two plots of Figure~\ref{fig:impedancematched}. Clearly, the
simulated wave packet travels left across the $x=0.5$ interface
without any reflected wave and expands as it enters the region of
higher wave speed.  In further (unreported) numerical experiments, we
have noticed changes in the discrete wave speed depending on the
space-time mesh. For example, the wave speed differs if one uses
uniform space time meshes with positively sloped diagonals only or
negatively sloped diagonals only. Such wave speed differences appear
to approach to zero slowly as $h$ is made smaller. High order methods
may be needed to reduce these dispersive errors.

\begin{figure}[t]  
  \begin{center}
  \begin{subfigure}{\textwidth}
    \centering         
    \includegraphics[trim=0.5in 2in 1in 2.in, clip, width=0.8\textwidth]
           {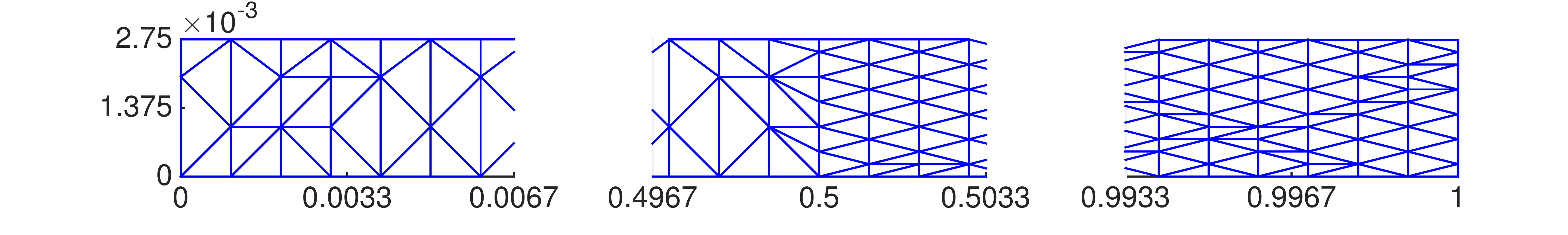}
    \caption{Left, middle and right parts of the tent pitched mesh on one time slab}
    \label{fig:mesh2}
  \end{subfigure}
  \\
  \begin{subfigure}{\textwidth}
    \centering
    \includegraphics[width=0.9\textwidth]{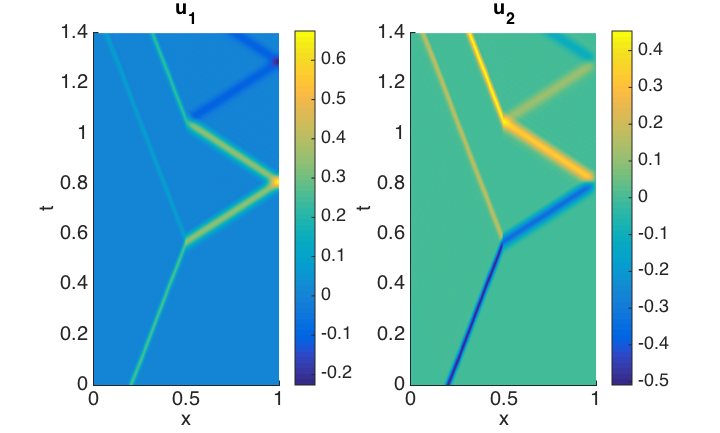}
    \caption{Space-time plot of the two solution components computed
      by explicit tent pitching}
    \label{fig:patch2}
  \end{subfigure}
  \caption{Case of reflected and transmitted waves}
  \label{fig:mismatched}
  \end{center}
\end{figure}

Our next and final example involves an interface where we expect both
reflection and transmission. We set $c=1$ and 
\[
\kappa_1 = 
\left\{
  \begin{matrix}
    4, &&  0< x < 1/2,
    \\
    1/2, &&  1/2 < x < 1,
  \end{matrix}
\right.
\qquad
\kappa_2 = 
\left\{
  \begin{matrix}
    1, &&  0< x < 1/2,
    \\
    1/2, &&  1/2 < x < 1,
  \end{matrix}
\right.
\]
Both the wave speed and the impedance jumps from the left region to
the right region (from 0.5 and 4 to 2 and 1, respectively). We set $f$
and initial data as in the last simulation by~\eqref{eq:16}, but in
order to impose Dirichlet boundary condition, we set $z_0=z_1=0$.
This time, instead of using a non-uniform mesh, we use a spatially
uniform mesh of $h=10^{-3}$ and let the tent pitching algorithm
adjust~$k$ to satisfy the CFL condition~\eqref{eq:cfl-2} in each
tent. We found that the mesh obtained, displayed in
Figure~\ref{fig:mesh2}, while not ideal due to the thin triangles, is
adequate for the simulation. (Better tent pitched meshes can be
obtained using non-uniform spatial mesh spacing, as we saw in the
previous example and Figure~\ref{fig:mesh1}.) The solution components
$u_1$ and $u_2$ obtained from the simulation are displayed in
Figure~\ref{fig:patch2}. The computed waves are transmitted as well as
reflected both from the interface and the Dirichlet boundaries. The
expected features of the solution are therefore recovered by the
method.

\end{document}